\documentclass[11pt,oneside]{amsart}
\usepackage{amssymb, amsmath, amsthm, stmaryrd}

\usepackage{epsfig}
\usepackage{graphicx}
\usepackage{color}
\usepackage{enumerate}

\usepackage{fullpage}
\usepackage{pinlabel}
\usepackage[hidelinks]{hyperref}
\usepackage{amsrefs}

\BibSpec{software}{%
+{ }{\PrintAuthors} {author}
+{, }{\textit} {title}
+{. }{Edited by \PrintAuthors} {editor}
+{. }{ } {journal}
+{. }{ } {url}
}

\theoremstyle{plain}
\newtheorem{theorem}{Theorem}[section]

\newtheorem*{theorem*}{Theorem}

\newtheorem{proposition}[theorem]{Proposition}
\newtheorem{corollary}[theorem]{Corollary}
\newtheorem{lemma}[theorem]{Lemma}

\newtheorem*{Thm-net equiv add}{Theorem \ref{net equiv add}}
\newtheorem*{Thm-equiv compare small}{Theorem \ref{thm: equiv compare small}}
\newtheorem*{Thm-counting lower bd}{Theorem \ref{counting lower bd}}
\newtheorem*{Thm-add subadd}{Theorem \ref{add subadd}}
\newtheorem*{Thm-super-add}{Theorem \ref{super-add}}
\newtheorem*{Thm-equivupper}{Theorem \ref{equiv-upper}}
\newtheorem*{corollary-thin sum}{Corollary \ref{thin sum}}
\newtheorem*{theorem-factor uniqueness}{Theorem \ref{factor uniqueness}}

\theoremstyle{definition}
\newtheorem{example}[theorem]{Example}
\newtheorem{remark}[theorem]{Remark}

\newtheorem{definition}[theorem]{Definition}

\theoremstyle{definition}

\newcommand{\N}{\mathbb N}

\renewcommand{\H}{\mathbb H}

\newcommand{\nil}{\varnothing}

\newcommand{\wihat}{\widehat}
\newcommand{\defn}[1]{\textbf{#1}}
\newcommand{\boundary}{\partial}
\newcommand{\mc}[1]{\mathcal{#1}}
\newcommand{\ob}[1]{\overline{#1}}

\renewcommand{\b}{\frak{b}}
\renewcommand{\t}{\frak{t}}

\newcommand{\g}{\frak{g}}

\newcommand{\weight}{\omega}
\newcommand{\X}{x_\weight}
\newcommand{\netX}{\operatorname{net}\X}
\newcommand{\netp}{\operatorname{net}\iota}
\newcommand{\cpt}{\sqsubset}

\newcommand{\WeightSet}{\N^\infty_2}

\newcommand{\co}{\mskip0.5mu\colon\thinspace}

\begin{document}


   \title[]{Equivariant Heegaard genus of reducible 3-manifolds}
   \author{Scott A. Taylor}

  \begin{abstract}
  The equivariant Heegaard genus of a 3-manifold $W$ with the action of a finite group $G$ of diffeomorphisms is the smallest genus of an equivariant Heegaard splitting for $W$. Although a Heegaard splitting of a reducible manifold is reducible and although if $W$ is reducible, there is an equivariant essential sphere, we show that equivariant Heegaard genus may be super-additive, additive, or sub-additive under equivariant connected sum. Using a thin position theory for 3-dimensional orbifolds, we establish sharp bounds on the equivariant Heegaard genus of reducible manifolds, similar to those known for tunnel number. Along the way, we make use of a new invariant for $W$ which is much better behaved under equivariant sums. 
   \end{abstract}

   \date{\today}
\thanks{}

   \maketitle

\section{Introduction}
Throughout, let $W$ be a compact, connected, oriented 3-manifold having no spherical boundary components. Let $G$ be a finite group of orientation-preserving diffeomorphisms of $W$.  A subset $X \subset W$ is \defn{equivariant} if for all $g \in G$, either $g(X) \cap X = \nil$ or $g(X) = X$. It is \defn{invariant} if $g(X) = X$ for all $g \in G$. Famously, the Equivariant Sphere Theorem \cite{MY-EST} says that if $W$ is reducible, then there exists an equivariant essential sphere. Consequently, if $W$ is reducible, then there exists an equivariant system of essential spheres $S \subset W$ giving a factorization of $W$ into irreducible 3-manifolds. The (disconnected) manifold $W|_S$ created by cutting open along $S$ and gluing in 3-balls is said to be obtained by \defn{equivariant surgery} on $S$. We extend the action of $G$ across the 3-balls.

Closed 3-manifolds also have equivariant Heegaard splittings \cite{Zimmermann96}. (A Heegaard splitting for a disconnected manifold is the union of Heegaard splittings of its components.) The \defn{equivariant Heegaard genus} $\g(W;G)$ is the minimum genus of an \emph{equivariant} Heegaard \emph{splitting} for $W$. (This is the same as the minimal genus of an \emph{invariant} Heegaard \emph{surface} for $W$.) If $G$ is trivial, then $\g(W) = \g(W;G)$ is the \defn{Heegaard genus} of $W$. In general, $\g(W;G) \geq \g(W)$. A \defn{reducing sphere} for a Heegaard surface $H \subset W$ is a sphere intersecting $W$ in a single essential simple closed curve. If $W$ is reducible, Haken's lemma states that every Heegaard surface for $W$ admits a reducing sphere \cite{Haken}. In particular, if $W$ is reducible, every invariant Heegaard surface for $W$ admits a reducing sphere.

\textbf{Question 1:} (Equivariant Haken's Lemma) If $W$ is reducible and if $H$ is an invariant Heegaard surface for $W$, must there exist an \emph{equivariant} essential sphere in $W$ that is a reducing sphere for $H$?

One important consequence of Haken's lemma is that Heegaard genus of 3-manifolds is additive under connected sum. That is, $\g(W) = \g(W|_S)$. We could also ask:

\textbf{Question 2: } (Additivity of Equivariant Heegaard Genus) If $W$ is reducible and if $S \subset W$ is an invariant system of reducing spheres such that each component of $W|_S$ is irreducible, is $\g(W;G) = \g(W|_S; G)?$

A positive answer to Question 1 implies a positive answer to Question 2. Both questions have a positive answer in the case when $G$ acts freely on $W$ (Corollary \ref{free action} below). Another reason to think that the answer to both questions might be ``yes,'' is that recently Scharlemann \cite{Sch-Haken} showed that given a collection of pairwise disjoint essential spheres $S \subset W$, there is an isotopy of any Heegaard surface $H$ for $W$ so that each component of $S$ is a reducing sphere for $H$. In particular, if $S$ is as in Question 2 and if $H \subset W$ is an invariant Heegaard surface, then $H$ can be isotoped so that each component of $S$ is a reducing sphere for $H$. Of course, the isotopy may destroy the fact that $H$ is invariant.

In fact, the answers to Questions 1 and 2 are negative.  Throughout, $|X|$ is either the cardinality or the number of components of $X$. 

\begin{theorem*}
Equivariant Heegaard genus can be sub-additive, additive, or super-additive. In particular,
\begin{enumerate}
\item If $k \geq 2$ is large enough, then  there exists a reducible $W$, an invariant essential sphere $S$ for $W$ dividing $W$ into two irreducible factors, and a cyclic group $G$ of diffeomorphisms of $W$ having order $k$ such that
\[
\g(W;G) < \g(W|_S; G).
\]
\item For $k \geq 2$, there exists a reducible $W$, an invariant essential sphere $S$ for $W$ dividing $W$ into two irreducible factors, and a cyclic group $G$ of diffeomorphisms of $W$ having order $k$ such that
\[
\g(W;G) = \g(W|_S; G).
\]
\item For $k \geq 2$, there exists a reducible $W$ having a finite cyclic group of diffeomorphisms $G$ of order $k$, and an invariant essential sphere $S$ dividing $W$ into two irreducible manifolds, such that
\[
\g(W;G)  = \g(W|_S; G)  + k - 1.
\]
Furthermore, $\g(W;G)$ can be arbitrarily high for fixed $k$.
\end{enumerate}
\end{theorem*}

Conclusions (1) and (2) are the content of Theorems \ref{subadd} and \ref{add}. Conclusion (3) is the content of Theorem \ref{super-add}. Consequently, no invariant Heegaard surface $H$ for the manifolds $W$ in Conclusions (1) or (3) admits an equivariant reducing sphere. Perhaps not surprisingly we use facts about the non-additivity of tunnel number to prove those theorems. And so we are left with:

\textbf{Question 3:} What can we say about the relationship between $\g(W;G)$ and $\g(W|_S;G)$? 

We establish an upper bound for $\g(W;G)$ relative to a system of separating summing spheres. We also note that Rieck and Rubinstein \cite{RR} have produced an upper bound on $\g(W;G)$ in the case when $G = \{1, \tau\}$, with $\tau$ an involution. However, their upper bound does not apply in the case when $W$ is reducible. Theorem \ref{super-add} shows that the following inequality for $c = 1$ is sharp; the other inequality is likely also sharp. 

\begin{Thm-equivupper}
For any invariant system of summing spheres $S \subset W$ such that each component of $S$ is separating,
\[
\g(W;G) \leq \g(W|_S;G) + (c(|G|+1) - 2)(\big|W|_S\big|-1)
\]
where $c = 1$ if every point of $W$ has cyclic stabilizer and $c = 2$ otherwise.
\end{Thm-equivupper}

Turning to lower bounds, we show:

\begin{Thm-counting lower bd}
Suppose that $W$ is closed and connected and that $G$ does not act freely. If $S \subset W$ is an invariant system of summing spheres, then 
\[
\g(W;G) \geq |G|(\nu/12 - \mu/2) + |S| + 1,
\]
where $\nu$ is the number of orbits of the components of $(W|_S)$ that are not $S^3$ or lens spaces and $\mu$ is the number of orbits of components of $W|_S$ that are homeomorphic to $S^3$.
\end{Thm-counting lower bd}

This bound is somewhat weak since $\g(W;G) \geq \g(W) = \g(W|_S) \geq 2\nu$. However, it does give some indication of how the order of $G$ affects the equivariant Heegaard genus. Additionally, Examples \ref{ex 1} and \ref{exam 2} suggest we are unlikely to be able to improve the lower bound  without additional hypotheses. In a different direction, we show that equivariant Heegaard genus is additive or super-additive when the factors are equivariantly comparatively small. A 3-manifold $W$ is \defn{equivariantly comparatively small} relative to the action of a finite group of diffeomorphisms $G$ if every equivariant essential surface $F \subset W$ has  $\g(F) > \g(W;G)$.

\begin{Thm-equiv compare small}
Suppose that when $S \subset W$ is an invariant system of summing spheres, then every component of $W|_S$ is equivariantly comparatively small. Then:
\[
\g(W;G) \geq \g(W|_S; G).
\]
\end{Thm-equiv compare small}

As a by-product of our methods, we produce a new invariant which we call \defn{equivariant net Heegaard characteristic} $\netX(W;G)$. It is similar to equivariant Heegaard genus, but its additivity properties are entirely understood, at least when every component of $S$ is separating:
\begin{Thm-net equiv add}
There exists an invariant system of summing spheres $S \subset W$ such that
\[
\netX(W;G) \geq \netX(W|_S; G) - 2|S|,
\]
and $W|_S$ is irreducible. If each component of $S$ is separating, then equality holds.
\end{Thm-net equiv add}

Like equivariant Heegaard genus, $\netX(W;G)$ is defined by minimizing a certain quantity over a surface. In this case, however, it is not an equivariant Heegaard splitting but an equivariant generalized Heegaard splitting. Theorem \ref{net equiv add} arises from the fact that equivariant spheres arise as thin surfaces. We explain this in Section \ref{sec: thin} using orbifolds. Part of the point of this paper is to advertise the very useful properties of the invariant $\netX$, which is closely related to the invariant ``net extent'' of \cites{TT2, TT3}.

\subsection*{Orbifolds and Outline}

Three-dimensional orbifolds are the natural quotient objects resulting from finite diffeomorphism group actions on 3-manifolds. At least since Thurston's work on geometrization \cite[Chapter 13]{Thurston}, 3-dimensional orbifolds have been the essential tool for understanding finite group actions on manifolds; they are also important and interesting in their own right.  A closed orientable 3-orbifold is locally modelled on the quotient of a 3-ball by a finite group of orientation preserving diffeomorphisms. The \defn{singular set} is the set of points which are the images of those points with nontrivial stabilizer group. A closed, orientable 3-orbifold can thus be considered as a pair $(M,T)$ where $M$ is a closed orientable 3-manifold and $T \subset M$ is a properly embedded trivalent spatial graph with integer edge weights $\weight \geq 2$. If $e$ is an edge of $T$ with $\weight(e) = w$, then we say that $e$ has weight $w$. If $v \in T$ is a vertex with incident edges having weights $a,b,c$, we let $\X(v) = 1 - (1/a + 1/b + 1/c)$.

When $(M,T)$ is a closed, 3-dimensional orbifold, for each vertex $v$ we have $\X(v) < 0$. This occurs if and only if, up to permutation, $(a,b,c)$ is one of $(2,2,k)$, $(2,3,3)$, $(2,3,4)$, and $(2,3,5)$ for some $k \geq 2$. Conversely, every trivalent graph $T$ properly embedded in a compact (possibly with boundary) 3-manifold and having edge weights satisfying $\X(v) < 0$ at each trivalent vertex $v$ describes a compact, orientable 3-orbifold. We work somewhat more generally by taking edge weights in the set $\WeightSet = \{n \in\N : n \geq 2\} \cup \{\infty\}$. We define $1/\infty = 0$. This is similar to what is allowed by the software Orb \cite{Heard} for studying the geometry of 3-orbifolds. We consider the edges of weight $\infty$ as consisting of points belonging to a knot, link, or spatial graph in $M$ on which $G$ acts freely. Allowing such weights does not create any additional difficulties and means that our methods may be useful for studying symmetries of knots, links, and spatial graphs. Additionally, some of our bounds are achieved only when there are edges of infinite weight; as these bounds are asymptotically achieved when edges have only finite weights, it is conceptually clearer to allow edges of infinite weight. Henceforth, we define ``3-orbifold'' as follows. 

\begin{definition}
A \defn{3-orbifold} is a pair $(M,T)$ where $M$ is a compact, orientable 3-manifold, possibly with boundary and $T \subset M$ is a properly embedded trivalent graph such that each edge $e$ has a weight $\weight(e) \in \WeightSet$. We also require that at every trivalent vertex $v$, $\X(v) < 0$.
\end{definition}

When studying symmetries of knots or spatial graphs in a 3-manifold, vertices of degree 4 or more may arise, as may vertices incident to edges of infinite weight. Our framework can handle such vertices by considering them as boundary components. Our approach should be contrasted with the usual methods of proving Haken's lemma which involve intricate methods of controlling the intersections between a Heegaard surface and an essential sphere. We replace those arguments with appeals to machinery that guarantee that essential spheres show up as thin surfaces in a certain type of thin position. The basic structure of our arguments is as follows. 

In Section \ref{gn}, we define orbifold Heegaard splittings; show that they are the quotients of invariant Heegaard surfaces of $W$ by $G$; and show that each orbifold Heegaard splitting of the quotient orbifold lifts to a Heegaard splitting of $W$. We also review the correspondence between equivariant essential spheres and orbifold reducing spheres. Section \ref{Addsubadd} uses the correspondence to prove that equivariant Heegaard genus can be sub-additive or additive. Section \ref{sec: thin} adapts Taylor--Tomova's version of thin position to orbifolds. There we define the ``net Heegaard characteristic'' $\netX$ of an orbifold and establish its correspondence with equivariant net Heegaard characteristic. The equivariant net Heegaard characteristic of $W$ is bounded above by $2\g(W;G) - 2$.  We also prove Theorem \ref{net equiv add}. The key idea is to use Taylor-Tomova's partial order on the orbifold version of Scharlemann-Thompson's generalized Heegaard splittings for 3-manifolds. The minimal elements of this partial order are called ``locally thin.'' We show that $\netX$ is non-increasing under the partial order and that orbifold reducing spheres show up as thin surfaces in the generalized Heegaard splittings. These two properties make $\netX$ additive under orbifold-sums. Understanding the net Heegaard characteristic of the result of decomposing the orbifold by essential spheres allows us to produce lower bounds on $\netX$, and thus on $\g(W;G)$. Section \ref{sec: upper} establishes upper bounds on $\g(W;G)$ by adapting the amalgamation of generalized Heegaard splittings to orbifolds. Although, in principle, this should be straightforward to those who understand amalgamation of generalized Heegaard splittings, significant issues arise; issues which in some sense seem to characterize the non-additivity of equivariant Heegaard genus.

\subsection*{Acknowledgments} I am grateful to Marc Lackenby for suggesting the idea to apply the techniques of \cite{TT3} to orbifolds and equivariant Heegaard genus. Daryl Cooper, Yo'av Rieck, Marty Scharlemann, and an anonymous referee provided helpful comments. Finally, I'm grateful to Maggy Tomova for our long-time collaboration giving rise to the work that this paper is based on.

\section{General Notions, Orbifold Sums, and Orbifold Heegaard surfaces}\label{gn}
Throughout all manifolds and orbifolds are orientable. If $X$ is a topological space, we let $Y \cpt X$ mean that $Y$ is a path component of $X$. We let $X \setminus Y$ denote the complement of an open regular neighborhood of $Y$ in $X$. 

We refer the reader to \cites{Boileau, Cooper} for more on orbifolds. Suppose that $(M,T)$ is an orbifold. A properly embedded orientable surface $S \subset M$, transverse to $T$, naturally inherits the structure of a 2-suborbifold; that is a surface locally modelled on the quotient of a disc or half disc by a finite group of orientation-preserving isometries. We call the points $S \cap T$ the \defn{punctures} of $S$. If $p \in S \cap T$ is a puncture its \defn{weight} $\weight(p)$ is the weight of the edge intersecting it. The \defn{orbifold characteristic} of $S$ is defined to be:
\[
\X(S) = -\chi(S) + \sum\limits_{p \in S \cap T}\Big(1 - \frac{1}{\weight(p)}\Big) 
\]
This is the negative of the orbifold Euler characteristic of $S$. Observe that if $\pi\co R \to S$ is an orbifold covering map of finite degree $d$, then $\X(R) = d\X(S)$. An orbifold is \defn{bad} if it is not covered by a manifold and \defn{good} if it is. The bad 2-dimensional orbifolds are spheres that either have a single puncture of finite weight or have two punctures of different weight. 2-orbifolds that are spheres with three punctures are \defn{turnovers}. A good  connected 2-orbifold $S$ is \defn{spherical} if $\X(S) < 0$; \defn{euclidean} if $\X(S) = 0$; and \defn{hyperbolic} if $\X(S) > 0$. A compact, orientable 3-orbifold is good if it does not contain a bad 2-orbifold  \cite[Theorem 2.5]{Cooper}. 

\begin{definition}
A 3-orbifold $(M,T)$ is \defn{nice} if:
\begin{enumerate}
\item for each $S \cpt \boundary M$, $\X(S) \geq 0$ and if $S$ is a sphere then $|S \cap T| \geq 3$;
\item the pair $(M,T)$ has no bad 2-suborbifolds and no spheres that are once-punctured with infinite weight puncture;
\end{enumerate}
\end{definition}

The reason for forbidding bad 2-suborbifolds is that we cannot surger along them to create a valid 3-orbifold. The requirement that $\X(S) \geq 0$ for $S \cpt \boundary M$ helps with some of our calculations. The other requirements help with our use of the material from \cites{TT1, TT2}. We note that $(W, \nil)$ is a nice orbifold. (Recall  $\boundary W$ has no spheres by our initial definition of $W$.)

\subsection{Factorizations}
Petronio \cite{Petronio} proved a unique factorization theorem for closed good 3-orbifolds $(M,T)$ without nonseparating spherical 2-suborbifolds. We are working in a slightly more general context (because our orbifolds may not be closed, we allow good nonseparating spherical 2-suborbifolds, and we allow infinite weight edges). Nevertheless, we adopt Petronio's terminology and his results carry over to our setting as we now describe.

Suppose that $(M_1, T_1)$ and $(M_2, T_2)$ are distinct nice orbifolds. Let $p_1 \in M_1$ and $p_2 \in M_2$. We can perform a connected sum of $M_1$ and $M_2$ by removing a regular neighborhood of $p_1$ and $p_2$ and gluing the resulting 3-manifolds together along the newly created spherical boundary components; after gluing the corresponding sphere is a \defn{summing sphere}. To extend the sum to the pairs $(M_1, T_1)$ and $(M_2, T_2)$ we place conditions on the points. We require that they are either in the interiors of edges of the same weight or that they are on vertices with incident edges having matching weights. The gluing map is then required to match punctures to punctures of the same weight. 

When $p_1$ and $p_2$ are disjoint from $T_1$ and $T_2$, the sum is a \defn{distant sum} and we write $(M,T) = (M_1, T_1) \#_0 (M_2, T_2)$. When each $p_i$ is in the interior of an edge of $T_i$, it is a \defn{connected sum} and we write $(M,T) = (M_1, T_1) \#_2 (M_2, T_2)$. When each $p_i$ is a trivalent vertex, it is a \defn{trivalent vertex sum} and we write $(M,T) = (M_1, T_1) \#_3 (M_2, T_2)$. This sum is well-defined, for a particular choice of vertices $p_1, p_2$ and bijection from the ends of edges incident to $p_1$ to those incident to $p_2$. See \cite[Section 4]{Wolcott}. The pair $\mathbb{S}(0) = (S^3, \nil)$ is the identity for $\#_0$. The pair $\mathbb{S}(2) = (S^3, T)$ where $T$ is the unknot is the identity for $\#_2$. The pair $\mathbb{S}(3) = (S^3, T)$ where $T$ is a planar (i.e. trivial) $\theta$-graph, is the identity for $\#_3$. In each case, we allow $T$ to have whatever weights make sense in context. Note that when $k< i \leq j$, both factors of $\mathbb{S}(i) \#_k \mathbb{S}(j)$ are trivial. This means that some care is needed when discussing prime factorizations. 

Conversely, suppose that $(M,T)$ is a nice orbifold. Given a connected spherical 2-suborbifold $S \subset (M,T)$ we may split $(M,T)$ open along $S$ and glue in two 3-balls each containing a graph that is the cone on the points $S \cap T$. This operation is called \defn{surgery} along $S$. Observe that the result is still an orbifold and that if $M$ is connected, each component of $(M,T)|_S$ is incident to one or more \defn{scars} from the surgery (i.e. the boundaries of the 3-balls we glued in). If $S$ is such a sphere or the disjoint union of such spheres, we denote the result of surgery along (all components of) $S$ by $(M,T)|_S$. Each component of $S$ produces two scars in $(M,T)|_S$; we say those scars are \defn{matching}. If $S$ separates $M$, then surgery is the inverse operation to distant sum, connected sum, or trivalent vertex sum.

Since we will be dealing with reducible orbifolds, we need to work with slight generalizations of compressing discs. Suppose that $S \subset (M,T)$ is a surface. An \defn{sc-disc} for $S$ is a zero or once-punctured disc $D$ with interior disjoint from $T \cup S$, with boundary in $S \setminus T$, and which is not isotopic (by an isotopy everywhere transverse to $T$) into $S$. If $\boundary D$ does not bound a zero or once-punctured disc in $S$, then $D$ is a \defn{compressing disc} or \defn{cut disc} corresponding to whether $D$ is zero or once-punctured. A \defn{c-disc} is a compressing disc or cut disc. Otherwise, $D$ is a \defn{semi-compressing disc} or \defn{semi-cut disc} respectively. In other words, a semi-compressing disc or semi-cut disc is a zero or once-punctured disc with inessential boundary in the surface but which is not parallel into the surface. The \defn{weight} $\weight(D)$ of an sc-disc $D$ is equal to the weight of the edge of $T$ intersecting it and 1 otherwise. Compressing a surface using an sc-disc $D$ decreases $\X$ by $2/\weight(D)$.  If $\boundary M$ admits an sc-disc $D$, then $(M,T)\setminus D$ is the result of \defn{$\boundary$-reducing} $(M,T)$.

 If $S$ does not admit a c-disc, it is \defn{c-incompressible}. A \defn{c-essential surface} is a surface $S \subset (M,T)$ where each component is:
\begin{enumerate}
\item a c-incompressible surface,
\item not parallel in $M \setminus T$ into $\boundary (M\setminus T)$ (i.e. not \defn{$\boundary$-parallel}), and
\item not an unpunctured sphere bounding a 3-ball in $M \setminus T$.
\end{enumerate}

Observe that for a surface,  being parallel in $(M,T)$ to a component of $\boundary M$ is more restrictive than being parallel into $\boundary (M\setminus T)$. 

\begin{definition}
An orbifold $(M,T)$ is \defn{orbifold-reducible} if there exists a c-essential spherical 2-suborbifold $S \subset (M,T)$.
\end{definition}

\begin{definition}[{c.f. \cite{Petronio}}]\label{efficient}
Suppose that $(M,T)$ is a nice orbifold and that $S \subset (M,T)$ is a closed spherical essential 2-orbifold. Then $S$ is an \defn{system of summing spheres} if  $(M,T)|_S$ is orbifold-irreducible. A system of summing spheres $S$ is \defn{efficient} if, for $i \in \{0,2,3\}$, whenever an $\mathbb{S}(i)$ component of $(M,T)|_S$ contains an $i$-punctured scar, then it contains the matching scar.
\end{definition}

For a system of summing spheres $S \subset (M,T)$, we call $(M,T)|_S$, a \defn{factorization} of $(M,T)$ using $S$. The main difference between a factorization using an efficient system of summing spheres and a prime factorization is that a factor in an efficient factorization may contain a nonseparating sphere. By work of Petronio and Hog-Angeloni -- Matveev, in the absence of nonseparating spherical 2-orbifolds, up to orbifold homeomorphism for an efficient system of summing spheres $S$, both $S$ and its factorization $(M,T)|_S$ are unique; however, these homeomorphisms are not necessarily realizable by an isotopy in $(M,T)$. We note that neither the Petronio, Hog-Angeloni--Matveev results nor the following theorem (which is based on those) is trivial. For instance, as observed in \cite[Example 1.3]{Petronio}, the connected sum of any knot in $S^3$ with $(S^1 \times S^2, S^1 \times \text{(point)})$ is pairwise homeomorphic to $(S^1 \times S^2, S^1 \times \text{(point)})$. Consequently, $(S^1 \times S^2, S^1 \times \text{(point)})$ does not have a prime factorization. The crux of the existence of an efficient splitting system is deferred to Corollary \ref{thin sum} below and is a consequence of Taylor--Tomova's work on thin position. Based on the discussion in \cite[Section 3]{Petronio}, the ability of thin position to handle nonseparating spherical 2-suborbifolds seems to be an instance where the thin position techniques have an advantage over normal surface techniques. We relegate the proof to the appendix, since it is a slightly more elaborate version of Petronio's and Hog-Angeloni--Matveev's proofs.

\begin{theorem}[after Petronio, Hog-Angeloni--Matveev]\label{factor uniqueness}
Suppose that $(M,T)$ is a nice orbifold that is orbifold-reducible. Then there exists an efficient system of summing spheres $S \subset (M,T)$; indeed, any system of summing spheres contains an efficient subset. Furthermore, any two such systems  $S$, $S'$ are orbifold-homeomorphic, as are $(M,T)|_S$ and $(M,T)|_{S'}$.
\end{theorem}

The Equivariant Sphere Theorem is an important tool for studying group actions on 3-manifolds. The statement we use is inspired by \cite[Theorem 3.23]{Boileau} and the subsequent remark. 
\begin{theorem}[{Equivariant Sphere Theorem \cite{MY-EST} (c.f. \cite{Dunwoody})}]\label{est}
Suppose that $\rho\co (W, T') \to (M,T)$ is a regular orbifold covering, with $(W, T')$ and $(M,T)$ nice. Then $(W, T')$ is orbifold-reducible if and only if $(M,T)$ is orbifold reducible.
\end{theorem}

\subsection{Orbifold Heegaard Splittings}

As is well known, a Heegaard splitting of a closed 3-manifold is a decomposition of the 3-manifold into the union of two handlebodies glued along their common boundary. Every closed 3-manifold has such a decomposition. 3-manifolds with boundary have a similar decomposition into two compressionbodies glued along their positive boundaries. Zimmermann (e.g. \cites{Zimmermann92, Zimmermann96, Zimmermann-survey}) defined Heegaard splittings of closed 3-orbifolds and used them to study equivariant Heegaard genus. In his decompositions, an orbifold handlebody is an orbifold that is the quotient of a handlebody under a finite group of diffeomorphisms. He gives an alternative description in terms of certain kinds of handle structures \cite[Proposition 1]{Zimmermann96}. We adapt this latter definition to define orbifold compressionbodies. Petronio \cite{Petronio-complexity} also defines handle structures for 3-orbifolds; the definitions differ only on the definition of 2-handles. Zimmermann's definition has the advantage that handle structures can be turned upside down. Most of our definitions apply to graphs in 3-manifolds more generally, so we state them in that generality if possible.

\begin{definition}[Handle Structures]
A \defn{ball 0-handle} or \defn{ball 3-handle} is a pair $(W, T_W)$ where $W$ is a 3-ball and $T_W$ is the cone on a finite (possibly empty) set of points in $\boundary W$. We also call ball 0-handles and 3-handles, \defn{trivial ball compressionbodies}. We set $\boundary_+ W = \boundary W$ and $\boundary_- W = \nil$. A \defn{product 0-handle} or \defn{product 3-handle} is a pair $(W, T_W)$ pairwise homeomorphic to $(F \times I, p \times I)$ where $F$ is a closed orientable surface, $p \subset F$ is finitely (possibly zero) many points. We also call product 0-handles and product 3-handles, \defn{trivial product compressionbodies}. We set $\boundary_\pm W$ to be the preimage of $F \times \{\pm 1\}$. The \defn{attaching region} for a 0-handle is the empty set and the \defn{attaching region} for a 3-handle $(W, T_W)$ is $(\boundary_+ W, T_W \cap \boundary_+ W)$.  A \defn{trivial compressionbody} is either a trivial ball compressionbody or a trivial product compressionbody.

A \defn{1-handle} or \defn{2-handle} is a pair $(H, T_H)$ pairwise homeomorphic to $(D^2 \times I, p \times I)$ where $p$ is either empty or is the center of $D^2$. The \defn{attaching region} of a 1-handle is the preimage of $(D^2 \times \boundary I, p \times \boundary I)$. The attaching region of a 2-handle is $((\boundary D^2) \times I, \nil)$. 
\end{definition}

A \defn{vp-compressionbody} $(C, T_C)$  is the union of finitely many 0-handles and 1-handles so that the following hold:
\begin{enumerate}
\item The 0-handles are pairwise disjoint, as are the 1-handles.
\item 1-handles are glued along their attaching regions to the positive boundary of the 0-handles, and are otherwise disjoint from the 0-handles
\item If $(H, T_H)$ is a 1-handle such that one component $(D,p)$ of its attaching region is glued to a 0-handle $(W, T_W)$, and if $T_H \neq \nil$, then $p \in T_H \cap \boundary_+ W$. 
\item $C$ is connected.
\end{enumerate}
See Figure \ref{fig: vpcompbdy} for an example. The ``vp'' stands for ``vertex punctured'' and is used since drilling out vertices changes trivial ball compressionbodies with vertices into trivial product compressionbodies. If $(C, T_C)$ is a vp-compressionbody, we let $\boundary_- C$ be the union of $\boundary_- W$ over the 0-handles $(W, T_W)$ and we let $\boundary_+ C = \boundary C \setminus \boundary_- C$. Edges of $T_C$ disjoint from $\boundary_+ C$ are called \defn{ghost arcs}. Closed loops disjoint from $\boundary C$ are \defn{core loops}. Edges with exactly one endpoint on $\boundary_+ C$ are \defn{vertical arcs} and edges with both endpoints on $\boundary_+ C$ are \defn{bridge arcs}. Dually, vp-compressionbodies may be defined as the union of 2-handles and 3-handles. Equivalently, if $(C, T_C)$ is a connected pair with one component of $\boundary C$ designated as $\boundary_+ C$, then $(C, T_C)$ is a vp-compressionbody if and only if there is a collection of pairwise disjoint sc-discs $\Delta$ for $\boundary_+ C$ such that the result of $\boundary$-reducing $(C, T_C)$ along $\Delta$ is the union of trivial ball compressionbodies and trivial product compressionbodies. The collection $\Delta$ is a \defn{complete collection of sc-discs} for $(C, T_C)$ if it is pairwise nonparallel. If $T_C$ has weights such that $(C, T_C)$ is both a vp-compressionbody and an orbifold, then we call $(C, T_C)$ an \defn{orbifold compressionbody}.

\begin{figure}[ht!]
\labellist
\small\hair 2pt
\pinlabel {$\boundary_- C$} [l] at 188 2
\pinlabel {$\boundary_+ C$} [r] at 93 77
\endlabellist
\centering
\includegraphics[scale=0.8]{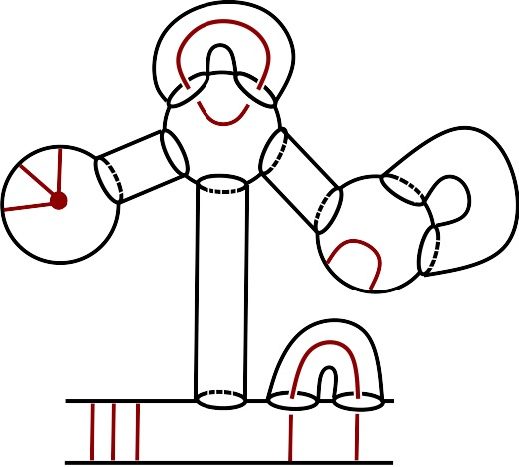}
\caption{An example of a vp-compressionbody $(C, T_C)$. It has one ghost arc, one core loop, one bridge arc, and three vertical arcs. The horizontal lines represent a closed, possibly disconnected surface $F$.}
\label{fig: vpcompbdy}
\end{figure}

For a nice orbifold $(M,T)$, an \defn{orbifold Heegaard surface} is a transversally oriented separating connected surface $H \subset (M,T)$ such that $H$ cuts $(M,T)$ into two distinct orbifold compressionbodies, glued along their positive boundaries. We define the \defn{Heegaard characteristic} of $(M,T)$ to be:
\[
\X(M,T) = \min\limits_H \X(H)
\]
where the minimum is over all orbifold Heegaard surfaces $H$ for $(M,T)$. The invariant $\X(M,T)$ is twice the negative of the  ``Heegaard number'' defined by Mecchia-Zimmerman \cite{MZ}.  Dividing by 2 would also make the comparison with \cites{TT2, TT3} easier. However, since orbifold Euler characteristic need not be integral, making that normalization would unpleasantly complicate some of the calculations in this paper. 

The \defn{equivariant Heegaard characteristic} of $W$ is \[\X(W;G) = \min\limits_H \X(H)\] where the minimum is taken over all \emph{invariant} Heegaard surfaces for $H$. Zimmermann proved the following when $W$ is closed; the proof extends to the case when $W$ has boundary as we explain. See also \cite{Futer} for a similar result related to strong involutions on tunnel number 1 knots. The statement is deceptively simple as it implies (and its proof relies on) the Smith Conjecture as well as a thorough understanding of group actions on 3-balls and products. See the remark on page 52 of \cite{Boileau}. 

\begin{lemma}[{after Zimmermann \cite{Zimmermann96}}]\label{quotient info}
Suppose that $W$ has orbifold quotient $(M,T)$. Every invariant Heegaard surface for $W$ descends to an orbifold Heegaard surface for $(M,T)$ and every orbifold Heegaard surface for $(M,T)$ lifts to an invariant Heegaard surface for $W$. Consequently,
\[
\X(W;G) = |G|\X(M,T).
\]
\end{lemma}
\begin{proof}
Suppose that $Y$ is a compressionbody (i.e. orbifold compressionbody with empty graph) and that $G$ is a finite group of orientation-preserving diffeomorphisms of $Y$. We show the quotient orbifold is an orbifold compressionbody. If $\boundary_- Y = \nil$ (i.e. $Y$ is a handlebody), then the quotient orbifold is an orbifold compressionbody by \cite{Zimmermann96}. In particular, the quotient of a 3-ball is an orbifold trivial ball compressionbody. Consider the case when $Y = F \times I$ for a closed, connected, oriented surface $F$. Since each element of $G$ is orientation-preserving, no element interchanges the components of $\boundary Y$. If $Y = S^2 \times I$, we cap off $\boundary_- Y$ with a 3-ball, extend the $G$-action across the 3-ball, and appeal to the 3-ball case to observe that the quotient orbifold is a trivial product orbifold compressionbody. Suppose that $F \neq S^2$. Let $DY$ be its double and observe that we can extend the action of $G$ to $DY$. The 3-manifold $DY$ is a Seifert fiber space with no exceptional fibers, namely $F \times S^1$. Consider an embedded torus $DQ$ that is the double of an essential spanning annulus $Q$ in $Y$. By \cite[Theorem 1.5]{Hass}, $DQ$ can be isotoped to $DQ'$ so that for each $g \in G$, $gDQ'$ is vertical in $DQ$ and is isotopic to $gDQ$. Since each $g \in G$ preserves each component of $\boundary Y$, there is an annulus $Q' \subset Y$, vertical in $Y$, such that for each $g \in G$, $gQ'$ is isotopic to $gQ$ by a proper isotopy in $Y$. Applying this observation to each annulus in a collection of spanning annuli in $Y$ cutting $Y$ into 3-balls, we can then conclude that the quotient orbifold of $Y$ by the action of $G$ is a trivial product orbifold, as desired. Finally, suppose that $\boundary_+ Y$ is compressible. By the Equivariant Disc Theorem \cite{MY-EDL}, $Y$ admits an equivariant essential disc. Boundary-reducing along this disc and inducting on $\X(\boundary_+ Y) - \X(\boundary_- Y)$ shows that the quotient orbifold is again an orbifold compressionbody. Our lemma then follows from the definition of invariant Heegaard surfaces, orbifold Heegaard surfaces, and the multiplicativity of orbifold characteristic under finite covers.
\end{proof}

\begin{corollary}\label{free action}
If $G$ acts freely on $W$, then there exists an invariant system of summing spheres $S \subset W$ such that $W|_S$ is irreducible and $\g(W;G) = \g(W|_S;G)$.
\end{corollary}
\begin{proof}
Let $(M,T)$ be the quotient orbifold of the action of $G$ on $W$. Assume the action is free so that $T =\nil$. There exists an efficient system of summing spheres $\ob{S} \subset M$. By Haken's Lemma, $\g(M) = \g(M|_{\ob{S}})$. Let $S$ be the pre-image of $\ob{S}$ in $W$. By Lemma \ref{quotient info} (after converting from Euler characteristic to genus), $\g(W;G) = \g(W|_S; G)$.
\end{proof}

\section{Examples of additivity and sub-additivity}\label{Addsubadd}
As an example of how to work with orbifold Heegaard surfaces, in this section we show that equivariant Heegaard genus can be both sub-additive and additive. The examples of super-additivity require different techniques. Our examples arise from cyclic branched covers of knots in $S^3$. We begin with a preliminary calculation.

Suppose that $K \subset S^3$ is a knot with weight $k \geq 2$. When we ignore the weight $k$, an orbifold Heegaard surface $H$ for $(S^3, K)$ of genus $g$ is a $(g,b)$-bridge surface where $b = |H \cap K|/2$ \cite{Doll}. Conversely, remembering the weight $k$, makes each $(g,b)$-bridge surface $H$ for $K$ into an orbifold Heegaard surface for $(S^3, K)$. If $H$ is such a surface, observe that
\[
\X(H) = 2(g + b - 1) - 2b/k.
\]
If $H$ is disjoint from $K$, then $K$ lies as a core loop of the handlebody to one side of $H$. The \defn{tunnel number} $\t(K)$ of $K$ is the minimum of $\g(H) - 1$ over all such $H$. If $H$ is any $(g,b)$-bridge surface for $K$, then by attaching tubes along the arcs of $K \setminus H$, we can convert $H$ into a $(g + b, 0)$-bridge surface for $K$. Consequently, $\t(K) \leq g + b - 1$ and $\X(S^3, K) \leq 2\t(K)$. The same technique shows that if, for fixed $g$, $H$ minimizes $b$, then $b \leq \b(K)$ where $\b(K)$ is the bridge number of $K$. (Recall $\b(K)$ is the minimal $b'$ such that $K$ admits a $(0,b')$-bridge surface.)

For $i = 1,2$, let $K_i \subset S^3$ be a knot of weight $k$ and let $H_i$ be a $(g_i, b_i)$-bridge surface for $K_i$. Assume that $b_1, b_2 > 0$. If we perform the connected sum of $K_1$ and $K_2$ using points of $K_1 \cap H_1$ and $K_2 \cap H_2$, then $H = H_1 \# H_2$ is a $(g_1 + g_2, b_1 + b_2 - 1)$-bridge surface for $K = K_1 \# K_2$. Note that if $S$ is the summing sphere for $(S^3, K)$ arising from our choice of connected sum, then $\X(S) = - 2/k$. Consequently,
\[
\X(H) = \X(H_1) + \X(H_2) - \X(\ob{S}).
\]

\begin{theorem}\label{subadd}
If $k \geq 2$ is large enough, then  there exists a reducible $W$, an invariant essential sphere $S$ for $W$ dividing $W$ into two irreducible factors, and a cyclic group $G$ of diffeomorphisms of $W$ having order $k$ such that
\[
\g(W;G) < \g(W|_S; G).
\]
\end{theorem}
\begin{proof}
We recall that there are many examples (e.g. \cites{Morimoto-degen, Nogueira, Schirmer}) of prime knots $K_1, K_2$ for which tunnel number is sub-additive. That is, $\t(K_1) + \t(K_2) - \t(K_1 \# K_2) \geq 1$. Recall $K= K_1 \# K_2$. Let $\ob{S}$ be the summing sphere. By \cite{Schubert}, $\b(K) = \b(K_1) + \b(K_2) - 1$. Suppose that $k > \b(K)$. Let $H_i$ be a $(g_i, b_i)$-bridge surface for $K_i$ such that $\X(H_i) = \X(S^3, K_i)$. According to the preamble, $\X(H_i) \geq 2\t(K_i) - 2b_i/k$. Thus,
\[\begin{array}{rcl}
\X(S^3, K_1) + \X(S^3, K_2) - \X(\ob{S}) &\geq& 2(\t(K_1) + \t(K_2)) - 2(b_1 + b_2 - 1)/k\\
&\geq& 2\t(K) + 2 -  2(b_1 + b_2 - 1)/k \\
&\geq& 2(\t(K) +1 - \b(K)/k)\\
&>& 2\t(K)\\
&\geq& \X(S^3, K).
\end{array}
\]
Passing to the $k$-fold cyclic branched cover $W$ over $K$, with $G$ the deck group and $S$ the preimage of $\ob{S}$,  produces the desired examples.
\end{proof}

\begin{remark}
From Morimoto's examples of tunnel number degeneration \cite{Morimoto-degen} we can see that in the proof of Theorem \ref{subadd}, we can choose the desired $K_1$ and $K_2$ so that $\b(K) \leq 6$. Thus, $k \geq 7$ suffices in the statement of Theorem \ref{subadd}.
\end{remark}

We now turn to examples of additivity. 

\begin{theorem}\label{add}
For $k \geq 2$, there exists a reducible $W$, an equivariant essential sphere $S$ for $W$ dividing $W$ into two irreducible factors, and a cyclic group $G$ of diffeomorphisms of $W$ having order $k$ such that
\[
\g(W;G) = \g(W|_S; G).
\]
\end{theorem}

\begin{proof}
Fix $k \geq 2$. For $i = 1,2$, let $K_i$ be a torus knot of type $(p,q)$ with $p, q$ relatively prime and $|p| > |q| \geq 5$. By \cite{Schubert}, each $\b(K_i) = |q|$. Also, each $K_i$ admits a $(1,1)$-bridge surface $H_i$. Thus, $\X(S^3, K_i) = \X(H_i) = 2(1 - 1/k)$. As in the preamble, let $K = K_1 \# K_2$ and $H = H_1 \# H_2$ so that $H$ is a $(2,1)$-bridge surface for $K$. Let $\ob{S}$ be the summing sphere. By the calculations in the preamble, 
\[
4 - \frac{2}{k} = \X(H) = \X(S^3, K_1) + \X(S^3, K_2) - \X(\ob{S}).
\]

Let $H'$ be a $(g,b)$-bridge surface for $K$ such that $\X(H') = \X(S^3, K)$. In particular, for that genus $g$, the number of punctures $2b$ is minimal. Also, since $\X(H') \leq \X(H)$:
\[
g \leq 2 - \frac{(k-1)(b-1)}{k} < 3 
\]

If $g = 0$, then by Schubert's result \cite{Schubert} that bridge number is $(-1)$-additive, we have $b = 2|q| - 1$. In which case,
\[
2 - 1/k = \X(H)/2 \geq \X(H')/2 \geq (2|q| - 1)(1 - 1/k) - 1
\]
a contradiction to the fact that $|q| \geq 5$ and $k \geq 2$.

If $g = 1$, our inequality shows that $b \leq 3$. Since $K$ is not the unknot, $b \geq 1$. Doll studied the situation when $g = 1$ and showed that (for our choice of $K_1$ and $K_2$), $b \geq |q| - 1$. (The proof for arbitrary $K_1$ and $K_2$ can be found in the solution \cite[Section 5]{Doll} of his Conjecture (1.1') for the case $g = 1$.) In which case,
\[
4 - 2/k = \X(H) \geq \X(H') = 2b(1 - 1/k) \geq 2(|q|-1)(1 - 1/k) \geq 8(1 - 1/k).
\]
But this contradicts the assumption that $k \geq 2$. 

If $g = 2$, and $b = 0$, then the tunnel number of $K$ would be 1, contradicting the fact that $K$ is composite. Thus, $g = 2$ and $b \geq 1$. Since $\X(H') \leq \X(H)$, we have $b = 1$ and so, $\X(H') = \X(H)$.

Thus, $\X(S^3, K) = \X(H') = \X(H) = \X(S^3, K_1) + \X(S^3, K_2) - \X(S)$. Passing to the $k$-fold cyclic branched cover $W$ over $K$, with $G$ the deck group and $S$ the preimage of $\ob{S}$,  produces the desired examples.
\end{proof}

\section{Orbifold thin position}\label{sec: thin}

Building on a long line of work concerning thin position, beginning with Gabai \cite{G3} and particularly including Scharlemann--Thompson \cites{ST-thingraph,ST-3mfld}, Hayashi--Shimokawa \cite{HS}, and Tomova \cite{Tomova}, Taylor--Tomova created a thin position theory for spatial graphs in 3-manifolds. In this section, we explain the minor adaptions needed to make it work for 3-orbifolds.

\begin{definition}[{Taylor--Tomova \cite{TT1}}]
Let $M$ be a compact, orientable 3-manifold and $T \subset M$ a spatial graph. A properly embedded closed surface $\mc{H} \subset (M,T)$ is a \defn{multiple vp-bridge surface} if the following hold:
\begin{enumerate}
\item $\mc{H}$ is the disjoint union of $\mc{H}^+$ and $\mc{H}^-$ where each of $\mc{H}^\pm$ are the union of components of $\mc{H}$;
\item Each component of $(M,T)\setminus \mc{H}$ is a vp-compressionbody;
\item $\mc{H}^+ = \bigcup \boundary_+ C$ and $\mc{H}^- \cup \boundary M = \bigcup \boundary_- C$ where each union is over all components $(C, T_C) \cpt (M,T) \setminus \mc{H}$. 
\end{enumerate}
When $\mc{H}$ has a transverse orientation, we can consider the \defn{dual digraph}; this is the digraph with a vertex for each vp-compressionbody and an oriented edge corresponding to each component of $\mc{H}$. We consider such $\mc{H}$ equipped with a transverse orientation such that the dual digraph is acyclic and each $(C,T_C) \cpt (M,T)\setminus \mc{H}$ is a cobordism from $\boundary_- C$ to $\boundary_+ C$. Equipped thusly, $\mc{H}$ is an \defn{oriented multiple vp-bridge surface}. We let $\H(M,T)$ be the set of oriented multiple vp-bridge surfaces up to isotopy transverse to $T$. When $(M,T)$ is an orbifold, we call the elements of $\H(M,T)$ \defn{multiple orbifold Heegaard surfaces}.
\end{definition}

After assigning an orientation, every orbifold Heegaard surface for a nice orbifold $(M,T)$ is an element of $\H(M,T)$. Since every spatial graph in a 3-manifold can be put into bridge position with respect to any Heegaard surface for the 3-manifold, $\H(M,T) \neq \nil$. For $\mc{H} \in \H(M,T)$, observe that $\mc{H}^- = \nil$ if and only if $\mc{H} = \mc{H}^+$ is connected.  If $T = \nil$, multiple vp-bridge surfaces induce the generalized Heegaard splittings of \cite{SSS}. The following lemma is a straightforward extension of Lemma \ref{quotient info}. 

\begin{lemma}\label{mvp preimage}
Suppose that $(M,T) \to (M', T')$ is a finite-sheeted orbifold cover and that $\mc{H} \in \H(M',T')$. Then the preimage of $\mc{H}$ is a multiple orbifold Heegaard surface for $(M,T)$. 
\end{lemma}

When $(M,T)$ does not contain any once-punctured spheres, Taylor and Tomova \cite{TT2} define an invariant called ``net extent'' on elements of $\H(M,T)$. We now adapt that invariant to the orbifold context.

\begin{definition}
Suppose that $(M,T)$ is a nice orbifold. For $\mc{H} \in \H(M,T)$, the \defn{net Heegaard characteristic} is
\[
\netX(\mc{H}) = \X(\mc{H}^+) - \X(\mc{H}^-).
\]
We define the \defn{net Heegaard characteristic} of $(M,T)$ to be:
\[
\netX(M,T) = \min \big\{ \netX(\mc{H}) : \mc{H} \in \H(M,T)\big\}.
\]
We define $\netX(W;G)$ similarly, but minimize only over \emph{invariant} $\mc{H} \in \H(W,\nil)$. 
\end{definition}

Proposition \ref{basic fact} below ensures that $\netX(M,T)$ is well-defined and that there exists $\mc{H} \in \H(M,T)$ with $\netX(\mc{H}) = \netX(M,T)$. The proof of the next lemma follows easily from Lemma \ref{mvp preimage}.

\begin{lemma}\label{netX equiv}
Suppose that $W$ has quotient orbifold $(M,T)$. Then
\[
\netX(W;G) = |G|\netX(M,T).
\]
\end{lemma}

In \cite{TT1}, Taylor and Tomova defined a set of operations on elements of $\H(M,T)$ and used them to define a partial order called \defn{thins to} and denoted $\mc{J} \to \mc{H}$. A minimal element in the partial order is said to be \defn{locally thin}. The operations involved in the definition of the partial order are all versions of the traditional ``destabilization'' and ``weak reduction'' of Heegaard splittings of 3-manifolds and ``unperturbing'' of bridge surfaces for links. We do not need the precise definitions of the operations in this paper, but we do need the following information. For our purposes, we group the operations into four categories (deferring to \cite{TT1} for precise definitions):
\begin{itemize}
\item[(I)] destabilization, meridional destabilization, $\boundary$-destabilization, meridional $\boundary$-destabilization, ghost $\boundary$-destabilization, meridional ghost $\boundary$-destabilization;
\item[(II)] unperturbing, undoing a removable arc;
\item[(III)] consolidation;
\item[(IV)] untelescoping.
\end{itemize}

The next lemma summarizes the key aspects of these operations. Pay particular attention to the disc $D$ used in the operations (I). In Lemma \ref{no incr}, we will need to analyze the effect of compressing along this disc on net Heegaard  characteristic. 

\begin{lemma}[Taylor-Tomova]\label{move facts}
The following hold:
\begin{itemize}
\item All of the operations listed in (I) involve replacing a thick surface $J \cpt \mc{J}^+$ with a new thick surface $H$ such that $H$ is obtained from $J$ by compressing along a compressing disc or cut disc $D$ and, if $\boundary D$ separates $J$, discarding a component. The component that is discarded is parallel to a surface obtained by tubing together some components of $\boundary M$ and vertices of $T$ along edges of $T$ disjoint from $\mc{H}$. 
\item All of the operations in (II) remove two punctures from a thick surface $J \cpt \mc{J}^+$.
\item Consolidation removes a thick surface and a thin surface from $\mc{J}$ that together bound a product vp-compressionbody with interior disjoint from $\mc{J}$.
\item Untelescoping replaces a thick surface $J \cpt \mc{J}^+$ with two new thick surfaces $H_1$ and $H_2$ and creates additional thin surfaces $F$. These surfaces arise from a pair of disjoint sc-discs on opposite sides of $J$. $H_1$ and $H_2$ are each obtained (up to isotopy) by compressing along one of the discs and $F$ is obtained by compressing along both.
\end{itemize}
\end{lemma}

\begin{lemma}\label{no incr}
Suppose that $(M,T)$ is an orbifold. As an invariant on $\H(M,T)$, $\netX$ is non-increasing under the partial order $\to$.
\end{lemma}
\begin{proof}
This lemma follows fairly directly from Lemma \ref{move facts}. Suppose that one of the moves in a thinning sequence replaces a thick surface $J$ with another thick surface $H$. Consider, first, the possibility that the move was of Type (I). Let $D$ be the disc we compress along, as in Lemma \ref{move facts} .  If $\boundary D$ is non-separating on $J$, we have
\[
\X(H) = \X(J) -2 + 2(1- 1/\weight(D)) \leq \X(J).
\]
 Thus, in such a case, the move does not increase $\netX$. If $\boundary D$ separates $J$, recall that after compressing we discard one component $J'$ of the resulting surface. That is, $H \cup J'$ is the result of the compression of $J$. We have
 \[
 \X(H) = (\X(J) -1 + (1 - 1/\weight(D)) + (-\X(J') - 1 + (1 - 1/\weight(D)).
 \]
 Suppose, in order to obtain a contradiction, that $\X(H) > \X(J)$.  Then 
 \[
0 \geq - 1 + (1 - 1/\weight(D)) >  \X(J') \geq -\chi(J') + (1 - 1/\weight(D))
 \]
 where the last inequality follows from the fact that $J'$ contains a scar from the compression by $D$. Since $J'$ is a closed surface, it must be a sphere. We recall from Lemma \ref{move facts}, that it is parallel to a certain surface $S$ obtained by tubing together components of $\boundary M$ and vertices of $T$ along edges of $T$ disjoint from $\mc{H}$. Let $\Gamma$ be the graph with a vertex for each component of $\boundary M$ and each vertex of $T$ that goes into the creation of $S$ and with edges the edges we tube along. Since $J'$ is a sphere and is parallel to $S$, $\Gamma$ must be a tree and each component of $\boundary M$ that is a vertex of $T$ is a sphere. If $\Gamma$ has an edge, there are at least two leaves and, by our definition of orbifold, each must be incident to at least two vertical arcs, giving $S$ at least 4 punctures. But in that case $\X(J') \geq 0$, a contradiction. Thus, $\Gamma$ is an isolated vertex; that is, $J'$ is parallel to either a component of $\boundary M$ or to a vertex of $T$. If it is a component of $\boundary M$, then $\X(J') \geq 0$, by hypothesis. Thus, $J'$ is parallel to a vertex $v$ of $T$. The vertex $v$ is trivalent with incident edges having weights $a,b,c$ and
 \[
0 <  -\X(J') - 1 + (1 - 1/\weight(D)) = -1 + 1/a + 1/b + 1/c - 1/\weight(D). 
  \]
 If $\weight(D) = 1$, then we have a contradiction. If $\weight(D) \neq 1$, then $D$ was a cut disc and so one of $a,b,c$ is equal to $\weight(D)$. Thus, in this case also, we have a contradiction. We conclude that $\X( H) \leq \X(J)$ and that none of the moves of Type (I) increase $\netX$. 
 
 The moves of Type (II) remove two punctures from a thick surface and so cannot increase $\netX$. If $H \cpt \mc{H}^+$ and $F \cpt \mc{H}^-$ are parallel, then $\X(H) = \X(F)$ and so consolidation does not change $\netX$. Finally, consider the operation of untelescoping. The three new surfaces $H_1$, $H_2$, and $F$ (as in the statement of Lemma \ref{move facts}) are all obtained by compressions along sc-discs for $J$. An easy computation shows that $\X(J) = \X(H_1) + \X(H_2) - \X(F)$ and so untelescoping also leaves $\netX$ unchanged.
\end{proof}

The next theorem is key to our endeavors. We have stated only what we need for this paper. We say that $\mc{H}^+$ is \defn{sc-strongly irreducible} if it is not possible to untelescope it (i.e. use move (IV) above).

\begin{theorem}[{Taylor--Tomova}]\label{locally thin}
Suppose that $T$ is a spatial graph in an orientable 3-manifold $M$ such that no component of $\boundary M$ is a sphere with two or fewer punctures and there is no once-punctured sphere in $(M,T)$. Then for every $\mc{J} \in \H(M,T)$ there exists a locally thin $\mc{H} \in \H(M,T)$ such that $\mc{J} \to \mc{H}$. Furthermore, for any locally thin $\mc{H}$ the following hold:
\begin{enumerate}
\item Each component of $\mc{H}^+$ is sc-strongly irreducible in $(M,T)\setminus \mc{H}^-$.
\item $\mc{H}^-$ is c-essential in $(M,T)$;
\item If $(M,T)$ contains a c-essential sphere with 3 or fewer punctures, then so does $\mc{H}^-$;
\item If $(M,T)$ is irreducible and if a component of $\mc{H}^+$ is a sphere with three or fewer punctures then $\mc{H} = \mc{H}^+$;
\item If $(C, T_C) \cpt (M,T)\setminus \mc{H}$ is a trivial product compressionbody, then $\boundary_- C \subset \boundary M$.
\end{enumerate}
\end{theorem}
\begin{proof}
We note that in \cite{TT1} saying that $T$ is irreducible, by definition, means that $(M,T)$ does not contain a once-punctured sphere. The existence of $\mc{H}$ given $\mc{J}$ is \cite[Theorem 6.17]{TT1}. Conclusions (1), (2),  and (5) can be found as Conclusions (2), (4) and (3) of \cite[Theorem 7.6]{TT1}, respectively. Conclusion (3) follows from \cite[Theorem 8.2]{TT1}. Conclusion (4) is similar to the proof of Conclusion (5) of \cite[Theorem 7.6]{TT1}. The details are similar to some of the arguments that follow, so we omit them here. 
\end{proof}

Examining the proof of Conclusion (3) in Theorem \ref{locally thin} provides us with more information about orbifolds. In particular, it produces another proof of the existence of systems of summing spheres. By Theorem \ref{factor uniqueness}, this also implies that efficient systems of summing spheres exist. The proof is nearly identical to that of \cite[Theorem 8.2]{TT1} and \cite[Proposition 5.1]{TT2}. Since the argument is completely topological, it is also the case that if $\mc{H}$ is locally thin, then $\mc{H}^-$ contains a collection of turnovers cutting $(M,T)$ into suborbifolds that contain no essential turnovers.
 For convenience, we provide the proof in the Appendix.

\begin{corollary}\label{thin sum}
Suppose that $(M,T)$ is a nice orbifold and let $\mc{H} \in \H(M,T)$ be locally thin. Then $\mc{H}^-$ contains a system of summing spheres for $(M,T)$.
\end{corollary}

\begin{theorem}\label{additivity}
Suppose that $(M,T)$ is a nice orbifold. Let $S \subset (M,T)$ be an efficient system of summing spheres. Then 
\[
\netX(M,T) \geq \netX((M,T)|_S) - \X(S).
\]
If each component of $S$ is separating, then equality holds.
\end{theorem}
\begin{proof}
We first show that $\netX(M,T) \geq \netX((M,T)|_S) - \X(S)$. Choose $\mc{J} \in \H(M,T)$ such that $\netX(\mc{J}) = \netX(M,T)$. This is possible by Proposition \ref{basic fact} below. By Theorem \ref{locally thin} and Lemma \ref{no incr}, there exists a locally thin $\mc{H} \in \H(M,T)$ with $\mc{J} \to \mc{H}$ and $\netX(\mc{H}) = \netX(M,T)$.  By Corollary \ref{thin sum}, there exists an efficient set of summing spheres $S \subset \mc{H}^-$. As we remarked, $S$ is unique up to orbifold homeomorphism, as is $(M,T)|_S$. When we split $(M,T)$ open along $S$, each component of $S$ is converted to two boundary components of $(M,T) \setminus S$. Boundary components are not included in the sum in the definition of $\netX$ and capping them off with trivial ball compressionbodies does not change that. The result follows.

When each component of $S$ is separating, the proof that $\netX(M,T) \leq \netX((M,T)|_S) - \X(S)$ is nearly identical to that of \cite[Theorem 5.5]{TT2}. In each component of $(M,T)|_S$ mark the points where sums will be performed. Since each component of $S$ is separating, the dual graph to $S$ is a tree. In each component $(M_i, T_i)$ of $(M,T)|_S$ choose $\mc{H}_i \in \H(M_i, T_i)$ such that $\netX(\mc{H}_i) = \netX(M_i, T_i)$. Again this is possible by Proposition \ref{basic fact} below. By transversality we may also assume each $\mc{H}_i$ is disjoint from the marked points. As in the proof of \cite[Theorem 5.5]{TT2}, we may reverse orientations on the $\mc{H}_i$ as necessary to ensure that their union with $S$ is an oriented multiple orbifold Heegaard surface for $(M,T)$. The desired inequality follows. 
\end{proof}

\begin{theorem}\label{net equiv add}
There exists an invariant system of summing spheres $S \subset W$ such that
\[
\netX(W;G) \geq \netX(W|_S; G) - 2|S|,
\]
and $W|_S$ is irreducible. If each component of $S$ is separating, then equality holds.
\end{theorem}
\begin{proof}
If $W$ is not orbifold-reducible, then the theorem is vacuously true. Otherwise, it is orbifold-reducible. By the Equivariant Sphere Theorem, the quotient orbifold $(M,T)$ is orbifold-reducible. We verify that $(M,T)$ is nice. Note that as $T$ is the singular set, no edge has infinite weight. Let $P \cpt \boundary M$ be a 2-sphere. If $\X(P) < 0$, then its pre-image in $\boundary W$ is the union of essential spheres, but there are none. Thus, $\X(P) \geq 0$. Since $(M,T)$ is covered by a manifold, there are no bad 2-suborbifolds. Consequently, $(M,T)$ is nice.

By Lemma \ref{netX equiv}, $\netX(W;G) = |G|\netX(M,T)$. By Theorem \ref{additivity}, there is an efficient system of summing spheres $\ob{S}$ for $(M,T)$ such that $\netX(M,T) \geq \netX((M,T)|_{\ob{S}})$ and equality holds if every component of $\ob{S}$ is separating. Let $S$ be the lift of $\ob{S}$ to $W$. If a component of $\ob{S}$ is non-separating, each component of its preimage in $W$ would be non-separating.

We have:
\[
\netX(W;G) = |G|\netX(M,T) \geq |G| \netX((M,T)|_{\ob{S}}) - |G|\X(\ob{S}) = \netX(W|_{S};G) - \X(S).
\]
If each component of $S$ is separating, then so is each component of $\ob{S}$ and equality holds. Let $W'$ be a component of $W|_S$. Its image in $(M,T)$ is a component of $(M,T)|_S$. If $W'$ were orbifold-reducible, then its image would be also, by the Equivariant Sphere Theorem. But this contradicts the fact that $\ob{S}$ is an efficient system of summing spheres. 

Suppose that some $S_0 \cpt S$ is inessential. Then it bounds a 3-ball $B \subset W$. Without loss of generality, we may assume that $S_0$ is innermost; i.e. the interior of $B$ is disjoint from $S$. The image of $B$ in $(M,T)$ is then the quotient of $B$ by its stabilizer. By \cite{Zimmermann96}, it is a trivial ball compressionbody and its boundary is inessential. This contradicts the fact that $\ob{S}$ is efficient. Thus, each component of $S$ is essential. \end{proof}

\begin{remark}
The proof of Theorem \ref{additivity} demonstrates the advantage that invariant multiple vp-bridge surfaces have over invariant Heegaard surfaces. Although there is no guarantee that if $W$ is reducible there is an equivariant sphere intersecting a minimal equivariant Heegaard splitting in a single closed loop, we can guarantee that there is an equivariant sphere showing up as a thin surface in an equivariant generalized Heegaard splitting of $W$. 
\end{remark}

\section{Lower bounds}\label{sec: lower}
The main purpose of this section is to find lower bounds on $\netX(M,T)$ for an orbifold $(M,T)$ and use that to prove Theorems \ref{counting lower bd} and \ref{thm: equiv compare small}. Along the way, we prove Proposition \ref{basic fact} which guarantees that $\netX(M,T)$ is well-defined and that there exists $\mc{H} \in \H(M,T)$ with $\netX(\mc{H}) = \netX(M,T)$. 

\subsection{Analyzing orbifold compressionbodies}

\begin{definition}
A \defn{lens space} is a closed 3-manifold of Heegaard genus 1, other than $S^3$ or $S^1 \times S^2$. A \defn{core loop} in a solid torus $D^2 \times S^1$ is a curve isotopic to $\{\text{point}\} \times S^1$. A \defn{core loop} in a lens space or $S^1 \times S^2$ is a knot isotopic to the core loop of one half of a genus 1 Heegaard splitting. A \defn{Hopf link} in $S^3$, $S^1 \times S^2$, or a lens space is a 2-component link such that there is a Heegaard torus separating the components and so that each component is a core loop for the solid tori on opposite sides of a Heegaard torus. A \defn{pillow} $(C, T_C)$ is a vp-compressionbody with boundary a 4-punctured sphere that is the result either of joining two (3-ball, arc) trivial ball compressionbodies by an unweighted 1-handle or joining two (3-ball, trivalent graph) compressionbodies by a weighted 1-handle. (See Figure \ref{fig: pillows}.) An orbifold that is a pillow or trivial ball compressionbody is \defn{Euclidean} if the boundary surface is. A \defn{Euclidean double pillow} is a pair $(S^3, T)$ with an orbifold bridge surface $H$ such that each of $(S^3, T)\setminus H$ is a Euclidean pillow. (The terminology stems from the fact that Euclidean orbifolds admit a complete metric locally modeled on Euclidean 2 or 3-space.)
\end{definition}

\begin{figure}[ht!]
\includegraphics[scale=0.8]{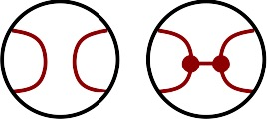}
\caption{The two types of pillow.}
\label{fig: pillows}
\end{figure}

If $(C, T_C)$ is the disjoint union of orbifold compressionbodies,  let $N(C,T_C) = \X(\boundary_+ C) - \X(\boundary_- C)$. Our key identity for an multiple orbifold Heegaard surface $\mc{H} \in \H(M,T)$ is:
\begin{equation}\label{Fundamental}
2\netX(M,T) - \X(\boundary M) = \sum\limits_{(C, T_C)} N(C, T_C).
\end{equation}
where we sum over the vp-compressionbodies $(C, T_C) \cpt (M,T) \setminus \mc{H}$. This follows immediately from the fact that each component of $\mc{H}^+$ and each component of $\mc{H}^-$ appears exactly twice as a boundary component of $(M,T)\setminus \mc{H}$.

\begin{lemma}\label{low N}
Suppose that $(C, T_C)$ is an orbifold compressionbody with no component of $\boundary_- C$ a once-punctured sphere. If $N(C, T_C) < 0$, then $(C, T_C)$ is a trivial ball compressionbody. Also, if $N(C, T_C) = 0$ and $\boundary_- C = \nil$, then $(C, T_C)$ is one of:
\begin{enumerate}
\item Euclidean trivial ball compressionbody 
\item Euclidean pillow
\item (solid torus, $\nil$)
\item (solid torus, core loop)
\end{enumerate}
\end{lemma}
\begin{proof}
Let $\Delta$ be a complete collection of sc-discs for $(C, T_C)$ such that $\boundary$-reducing $(C, T_C)$ along $\Delta$ results in trivial vp-compressionbodies $(C', T'_C)$.  Each disc of $\Delta$, leaves 2 ``scars'' on $\boundary_+ C'$. If $E$ is a scar, let $\weight(E) = 1$ if it is unpunctured and otherwise let $\weight(E)$ be the weight of the puncture. Let $(C_0, T_0) \cpt (C', T'_C)$. Let $N'(C_0, T_0)$ be equal to the sum of $N(C_0, T_0)$ with $\frac{1}{\weight(E)}$ for all scars $E$ on $\boundary_+ C_0$. Observe that
\[
N(C, T_C) = \sum\limits_{(C_0, T_0)} N'(C_0, T'_0)
\]
where the sum is taken over all $(C_0, T_0) \cpt (C', T')$. 

If $(C_0, T_0)$ is a product compressionbody then $N'(C_0, T_0) \geq N(C_0, T_0) = 0$ with equality if and only if every scar on $\boundary_+ C_0$ has weight $\infty$. Suppose that $(C_0, T_0)$ is a trivial ball compressionbody. 

\textbf{Case 1:} $T_0 = \nil$. 

If $\Delta = \nil$, then $(C, T_C) = (C_0, T_0)$ and $N(C, T_C) = -2$. Otherwise, by the choice of $\Delta$, $\boundary_+ C_0$ contains at least 2 scars, each of weight 1.  If it contains exactly 2, then $(C, T_C)$ is (solid torus, $\nil$). If it has at least 3 scars, then $N'(C_0, T_0) \geq 1$. 

\textbf{Case 2:} $T_0$ is an arc of weight $k$.

If $\Delta = \nil$, then $(C, T_C) = (C_0, T_0)$ and $N(C, T_C) = -\frac{2}{k} \geq -1$. If $N(C, T_C) = 0$, then $k = \infty$. If $\Delta \neq \nil$, then $\boundary_+ C_0$ contains at least one scar. By our choice of $\Delta$, either $(C, T_C)$ is (solid torus, core loop) or $\boundary_+ C_0$ contains at least 1 scar of weight 1. In which case, $N'(C_0, T_0) \geq 0$. Equality holds if and only if $k = 2$, there is exactly one scar and it has weight 1. 

\textbf{Case 3:} $T_0$ contains an interior vertex.

Note that $N(C_0, T_0) = \X(\boundary_+ C) = \X(v) < 0$ where $v$ is the internal vertex of $T_0$. If $\Delta = \nil$, then we have our result. If $\Delta \neq \nil$, then $\boundary_+ C_0$ contains at least one scar and it either has weight 1 or has weight equal to the weight of one of the punctures on $\boundary_+ C$. In which case, $N'(C_0, T_0) \geq 0$.  Equality holds only when there is exactly one scar, it contains a puncture, and the two punctures not contained in the scar both have weight 2.

This concludes our analysis of the individual cases and, in particular, we may assume that $\Delta \neq \nil$ and that $(C, T_C)$ is neither (solid torus, $\nil$) or (solid torus, core loop).  By our analysis, each component $(C_0, T_0) \cpt (C', T'_C)$ has $N'(C_0, T_0) \geq 0$. Thus, $N(C, T_C) \geq 0$. Suppose that $N(C, T_C) = 0$. Then $N'(C_0, T_0) = 0$ for each component of $(C', T'_C)$. Consequently, each component is one of:
\begin{itemize}
\item A product compressionbody such that every scar has infinite weight,
\item A trivial ball compressionbody containing an arc and with a single scar of weight 1,
\item A trivial ball compressionbody containing a trivalent vertex and with edges of weight $(2,2,k)$ with $k \geq 2$. It has a single scar of weight $k$.
\end{itemize}
The compressionbody $(C, T_C)$ can be reconstructed by attaching possibly weighted 1-handles to the scars on $(C', T')$. Thus our result holds if $N(C, T_C) \leq 0$. 
\end{proof}

We next have two propositions whose proofs are closely related. 
\begin{proposition}\label{basic fact}
Suppose that $(M,T)$ is a nice orbifold and that $\mc{H} \in \H(M,T)$. Then either $\netX(\mc{H}) \geq \frac{1}{2}\X(\boundary M)$ or $(M,T)$ is one of $\mathbb{S}(0)$, $\mathbb{S}(2)$, or $\mathbb{S}(3)$. Furthermore, there exists a locally thin $\mc{H} \in \H(M,T)$ with $\netX(\mc{H}) = \netX(M,T)$. If $(M,T)$ is an $\mathbb{S}(i)$, for $i \in \{0,2,3\}$, then any locally thin $\mc{H}$ is an $i$-punctured sphere.
\end{proposition}

\begin{proposition}\label{1/6 thm}
Suppose that $(M,T)$ is a nice, closed orbifold that is orbifold-irreducible and that $\mc{H} \in \H(M,T)$ is locally thin. Then either $\netX(\mc{H}) \geq 1/6$ or one of the following exceptional cases holds:
\begin{enumerate}
\item $\mathbb{S}(0)$, $\mathbb{S}(2)$, or $\mathbb{S}(3)$;
\item $M = S^3$ or a lens space and $T$ is  a core loop, Hopf link and $\mc{H}$ is an unpunctured torus; or
\item $(M,T)$ is a Euclidean double pillow and $\mc{H}$ is a four-punctured sphere.
\end{enumerate}
\end{proposition}

The remainder of the section is devoted to the proofs of these propositions.  A key bookkeeping device for a vp-compressionbody $(D, T_D)$ is its \defn{ghost arc graph}. This is the graph $\Gamma$ whose vertices are the components of $\boundary_- D$ and the vertices of $T_D$. The ghost arcs of $T_D$ are the edges. For example, if $(D, T_D)$ has a single ghost arc and it joins distinct components of $\boundary_- D$, then $\Gamma$ is a single edge. The key observation is that if $\boundary_+ D$ is a sphere, then $\Gamma$ is acyclic and if $\boundary_+ D$ is a torus, then $\Gamma$ contains at most one cycle. If it contains a cycle, then $\boundary_- D$ is the union of spheres (that is, it does not contain a torus).

Begin by assuming only that $(M,T)$ is a nice orbifold. Let $\mc{H} \in \H(M,T)$ be locally thin. Recall from Corollary \ref{thin sum} that $\mc{H}^-$ contains an efficient system of summing spheres $S$. Assume, for the time being, that $S = \nil$; equivalently, that $(M,T)$ is orbifold-irreducible. Since each component of $\mc{H}^-$ is c-essential in $(M,T)$, this also implies that no $S_0 \cpt \mc{H}^-$ is a sphere with $|S_0 \cap T| \leq 3$ and $\X(S_0) < 0$.

\textbf{Case 1: } Some $(C, T_C) \cpt (M,T) \setminus \mc{H}$ has $N(C, T_C) < 0$.  

By Lemma \ref{low N}, $(C, T_C)$ is a trivial ball compressionbody. Observe that $\boundary_+ C$ is a sphere with 0, 2, or 3 punctures. Let $(D, T_D) \cpt (M,T)\setminus \mc{H}$ be the other vp-compresionbody having $\boundary_+ D = \boundary_+ C$. If $\boundary_- D = \nil$, then $M = C \cup D$ and $(D, T_D)$ is also a trivial ball compressionbody. In this case, $(M,T)$ is either $\mathbb{S}(0)$, $\mathbb{S}(2)$, or $\mathbb{S}(3)$. Assume there exists $F \cpt \boundary_- D$.

Let $\Gamma$ be the ghost arc graph for $(D, T_D)$. As $\boundary_+ D$ is a sphere, $\Gamma$ is acyclic and the components of $\boundary_- D$ are all spheres. Since $(M,T)$ is nice, none of them are once-punctured. Since $S = \nil$, $F$ is at least thrice-punctured and has $\X(F) \geq 0$.  If $\Gamma$ contains an isolated vertex,  $(D, T_D)$ is a product.  Since $\mc{H}$ is locally thin, $F = \boundary_- D \subset \boundary M$.  If $T_C$ contains an interior vertex $v$, we must have $0 > \X(v) = \X(F) \geq 0$, a contradiction. If $T_C$ does not contain an interior vertex, then $F$ is twice-punctured, contradicting our definition of nice 3-orbifold. Thus, we may assume that $\Gamma$ does not have an isolated vertex. Since no component of $\boundary_- D$ is a twice-punctured sphere,  each leaf of $\Gamma$ is incident to at least two vertical arcs, so there is at most one leaf. Since $\Gamma$ is acyclic, this is a contradiction.  Consequently, $(M,T)$ is one of the exceptional cases in the statement of Proposition \ref{basic fact}. 

\textbf{Case 2:} Every $(C, T_C) \cpt (M,T) \setminus \mc{H}$ has $N(C, T_C) \geq 0$.  

By \eqref{Fundamental}, we see that $\netX(\mc{H}) \geq \X(\boundary M)/2$. Let $L$ be the product of all the finite weights on $T$. Note that for any $\mc{J} \in \H(M,T)$, the quantity $2L\netX(\mc{J})$ is an integer, as is $2L\X(\boundary M)/2$.  By Theorem \ref{locally thin}, for any $\mc{J} \in \H(M,T)$, there exists a locally thin $\mc{H} \in \H(M,T)$ with $\mc{J} \to \mc{H}$. By Lemma \ref{no incr}, $\netX(\mc{J}) \geq \netX(\mc{H})$. If $(M,T)$ is  one of the exceptional cases from Proposition \ref{basic fact}, then by the analysis in Case 1, $\mc{H}$ is connected and so $L\netX(\mc{J})$ is bounded below by a constant depending only on $(M,T)$. If $(M,T)$ is not one of the exceptional cases from Proposition \ref{basic fact}, then we see that $2L\netX(\mc{J}) \geq 2L\X(\boundary M)/2 \geq 0$. Thus, in either case, since the invariant $2L\netX$ defined on $\H(M,T)$ is integer-valued and bounded below by a number depending only on  it achieves its minimum on a locally thin element of $\H(M,T)$. That element also minimizes $\netX$. This concludes the analysis when $S = \nil$ for the proof of Proposition \ref{basic fact}.

Now suppose that $S \neq \nil$.  Expand $S$ to include all summing spheres in $\mc{H}^-$; continue to call it $S$. As we have observed previously,
\[
\netX(\mc{H}) = \netX(\mc{H} \setminus S) - \X(S).
\]
Let $(M_0, T_0) \cpt (M,T)|_S$ and let $\mc{H}_0 = (\mc{H}\setminus S) \cap M_0$. If $\netX(\mc{H}_0) < 0$, then $(M_0, T_0)$ is one of the exceptional cases from Proposition \ref{basic fact}. Since $\mc{H}$ is locally thin, each component of $S$ is essential, so $(M_0, T_0) \neq \mathbb{S}(0)$. If $(M_0, T_0) = \mathbb{S}(2)$, then at least one of the components $S'$ of $S$ used to sum with $(M_0, T_0)$ must be unpunctured. Thus, if $T_0$ has weight $k$, we have:
\[
\netX(\mc{H}_0) - \frac{1}{2} \X(S') = -\frac{2}{k} + 1 \geq 0.
\]
Similarly, if $(M_0, T_0)$ is an $\mathbb{S}(3)$, then at least one of the components $S' \cpt S$ used to sum with $(M_0, T_0)$ must be either unpunctured or twice punctured and with the weight of the punctures equal to the weight $c$ of one of the edges of $T_0$. In that case, letting $a,b$ be the weights of the other punctures,
\[
\netX(\mc{H}_0) - \frac{1}{2}\X(S') \geq 1 - (\frac{1}{a} + \frac{1}{b} + \frac{1}{c}) + \frac{1}{2}\cdot \frac{2}{c} \geq 0.
\]
Consequently, $\netX(\mc{H}) \geq \X(\boundary M)/2$, even in this situation. As before, the quantity $L\netX$ is an integer-valued invariant on $\H(M,T)$ bounded below by a constant depending only on $(M,T)$ and so, as before, there is a locally thin $\mc{H} \in \H(M,T)$ with $\netX(\mc{H}) = \netX(M,T)$. This concludes the proof of Proposition \ref{basic fact}. 

Henceforth, suppose that $(M,T)$ is closed and orbifold-irreducible and not one of the exceptional cases from Proposition \ref{basic fact}. By our previous remarks, this implies that $N(C, T_C) \geq 0$ for every $(C, T_C) \cpt (M,T)\setminus \mc{H}$. The dual digraph to $\mc{H}$ is acyclic, so it has at least one source and one sink. The sources and sinks are exactly those $(C, T_C) \cpt (M,T)\setminus \mc{H}$ with $\boundary_- C  = \nil$. Suppose that $(C, T_C)$ is one such.  Note that $N(C, T_C) = \X(\boundary_+ C)$. Let $(D, T_D) \cpt (M,T)\setminus \mc{H}$ be the other orbifold compressionbody with $\boundary_+ D = \boundary_+ C = H$.

Observe that $\X(H) \geq -\chi(H) + |H \cap T|/2$. Equality holds only if every puncture on $H$ has weight 2. Consequently, if $1/6 > \X(H)$, then $H$ is a sphere with $|H \cap T| \leq 4$. If $|H \cap T| \leq 3$, then by our analysis above $(M,T)$ is one of the exceptional cases from Proposition \ref{basic fact}. Consider, therefore, the case that $|H \cap T| = 4$. If at least one puncture does not have weight $2$, then $\X(H) \geq 1/6$, so assume that each puncture has weight 2. If $\boundary_- D = \nil$, then $\mc{H}$ divides  $(M,T)$ into two Euclidean pillows. Suppose $\boundary_- D \neq \nil$ and let $\Gamma$ be the ghost arc graph for $(D, T_D)$ as above. It is acyclic. Each component of $\boundary_- D$ is a sphere with at least three punctures, since $(M,T)$ is orbifold-irreducible and nice. Also $\boundary_- D \subset \boundary M$ since $M$ is closed. An isolated vertex of $\Gamma$ is a sphere incident to at least three vertical arcs and a leaf is a sphere incident to at least two vertical arcs. Since $H$ has four punctures, if $\Gamma$ has an isolated vertex, that vertex is the entirety of $\Gamma$ and it is incident to four vertical arcs. This implies $(D, T_D)$ is a product and contradicts local thinness of $\mc{H}$. Thus, $\Gamma$ has two leaves, each incident to two vertical arcs. At least one of those leaves $F$ is a component of $\boundary_- D$ (the other may be a vertex of $T$). Since $(M,T)$ is orbifold-irreducible, $\X(F) \geq 0$. Consequently, at least two of the arcs incident to $F$ have weight at least 3. At least one of those is a vertical arc, contradicting the fact that each puncture of $H$ has weight 2. Consequently, $N(C, T_C) = \X(H) \geq 1/6$. 

Since the dual digraph to $\mc{H}$ has at least one source and one sink either $(M,T)$ is one of the exceptional cases of Proposition \ref{basic fact}, or $\mc{H}$ is a four punctured sphere dividing $(M,T)$ into two Euclidean pillows, or $2\netX(\mc{H}) \geq 2\cdot (1/6)$. This concludes the proof of Proposition \ref{1/6 thm}.

Figure \ref{1/6 sharp} shows that our bound of 1/6 is asymptotically sharp.

\begin{figure}[ht!]
\labellist
\small\hair 2pt
\pinlabel $4$ [t] at 122 111
\pinlabel $2$ [bl] at 33 108
\pinlabel $3$ [r] at  217 110
\pinlabel $a$ [l] at 124 178
\pinlabel $2$ [bl] at 93 236
\pinlabel $2$ [br] at 149 236
\pinlabel $4$ [t] at 121 266
\pinlabel $2$ [r] at 91 142
\pinlabel $2$ [l] at 155 142
\endlabellist
\centering
\includegraphics[scale=0.8]{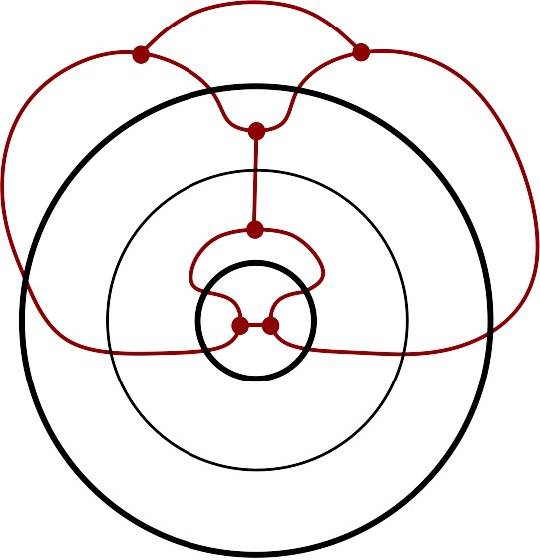}
\caption{An example of an orbifold with underlying 3-manifold $S^3$. The thick circles represent thick spheres and the thin circle is a thin sphere of a multiple vp-bridge surface $\mc{H}$. Arbitrary gluing maps preserving the punctures pointwise can be used along the thick spheres.  For $a \in \WeightSet$ we have $\netX(\mc{H}) = \frac{1}{6} + \frac{1}{a}$ and the orbifold characteristic of the thin sphere is $\frac{1}{6} - \frac{1}{a}$. Thus, for $a \geq 6$, $\mc{H}^-$ does not contain a spherical orbifold. As $a \to \infty$, we approach 1/6.}
\label{1/6 sharp}
\end{figure}

\subsection{Equivariant Heegaard genus and the order of the group}

\begin{theorem}\label{counting lower bd}
Suppose that $W$ is closed and connected and that $G$ does not act freely. If $S \subset W$ is an invariant system of summing spheres, then 
\[
\g(W;G) \geq |G|(\nu/12 - \mu/2) + |S| + 1,
\]
where $\nu$ is the number of orbits of the components of $(W|_S)$ that are not $S^3$ or lens spaces and $\mu$ is the number of orbits of components of $W|_S$ that are homeomorphic to $S^3$.
\end{theorem}

\begin{proof}
Let $(M,T)$ be the quotient orbifold and note that no edge of $T$ has infinite weight. Let $\ob{S}$ be the image of $S$ and note that each component of $(M,T)|_{\ob{S}}$ is orbifold-irreducible by the definition of ``system of summing spheres'' and Theorem \ref{est}. By Theorem \ref{factor uniqueness}, $\ob{S}$ contains an efficient subset. Let $S_0 \subset S$ be the preimage of a component of $\ob{S}$ that is not in our chosen efficient subset. Since $W$ is closed, each component of $W|_S$ that is not a component of $W|_{S_0}$ is a homeomorphic to $S^3$. Thus, passing from $S$ to $S\setminus S_0$ increases the right hand side of our inequality by $|G|/2 - 1$. Since $G$ does not act freely, it is not the trivial group. We conclude that it is enough to prove our result when $\ob{S}$ is efficient. Furthermore, by Theorem \ref{factor uniqueness}, we may prove it when $\ob{S}$ is \emph{any} efficient system of summing spheres for $(M,T)$, not merely the given one.

Recall that $\X(W;G) \geq |G|\netX(M,T)$. By Theorem \ref{locally thin} and Proposition \ref{basic fact}, there is a locally thin $\mc{H} \in \H(M,T)$ such that $\netX(\mc{H}) = \netX(M,T)$.  Furthermore, $\mc{H}^-$ contains an efficient set of summing spheres (Corollary \ref{thin sum}). Call them $\ob{S}$. Let $S$ be the preimage of $\ob{S}$ in $W$. By Theorem \ref{additivity}, we have
\[
\netX(M,T) = \netX((M,T)|_{\ob{S}}) - \X(\ob{S}). 
\]

Suppose some $(M_i, T_i)\cpt (M,T)|_{\ob{S}}$ is an $\mathbb{S}(0)$. Since $\ob{S}$ is efficient, the matching scar to every scar in $(M_i, T_i)$ is also in $(M_i, T_i)$. It follows that $(M,T)|_{\ob{S}} = (M_i, T_i)$. In which case, $T = \nil$, and $G$ acts freely, contrary to hypothesis. Henceforth, we may assume no $(M_i, T_i) \cpt (M,T)|_S$ is an $\mathbb{S}(0)$. Suppose that $(M_i, T_i)$ is an $\mathbb{S}(2)$. Then $\X(M_i, T_i) \geq -1$ and each component of the pre-image of $(M_i, T_i)$ in $W|_S$, is a copy of $S^3$. Similarly, if $(M_i, T_i)$ is an $\mathbb{S}(3)$, then $\X(M_i, T_i) \geq -1/2$ and each component of the pre-image of $(M_i, T_i)$ in $W|_S$ is also a copy of $S^3$. Conversely, by the classification of finite groups of diffeomorphisms of $S^3$,  if $W_i \cpt W|_S$ is a copy of $S^3$, then its image in $(M,T)|_S$ is an $\mathbb{S}(k)$ for some $k \in \{0,2,3\}$. Consequently, $\mu$ both the number of orbits of $S^3$ components of $W|_S$ and the number of $(M_i, T_i)$ that are $\mathbb{S}(2)$ or $\mathbb{S}(3)$. 

If $(M_i, T_i)$ is a ($S^3$, Hopf link), (lens space, core loop), (lens space, Hopf link), or a Euclidean double pillow, then $\X(M_i, T_i) = 0$. If a component $W_i$ of $W|_S$ covers $(M_i, T_i)$, then $W_i$ admits an invariant Heegaard torus, but no Heegaard sphere. Indeed, any $W_i$ that is a lens space has $\X(W_i) \geq 0$. Amalgamating a generalized Heegaard splitting of a 3-manifold produces a Heegaard surface and does not change $\X$. Thus, $\netX(W_i; G) \geq 0$, whenever $W_i$ is a lens space.

If $(M_i, T_i)$ is  neither an $\mathbb{S}(k)$ for $k \in \{0,2,3\}$ nor a ($S^3$, Hopf link), (lens space, core loop), (lens space, Hopf link), or a Euclidean double pillow, then by Proposition \ref{1/6 thm}, $\netX(M_i, T_i) \geq 1/6$. 

Consequently,
\[
\X(W;G) \geq |G|\netX(M,T) \geq |G|(\nu/6 - \mu - \X(\ob{S})) = |G|(\nu/6 - \mu) - \X(S). 
\]
Converting to genus, we have
\[
\g(W;G) \geq |G|(\nu/12 - \mu/2) + |S| + 1
\]
\end{proof}

\begin{example}\label{ex 1}
Consider a lens space $M$ containing an unknot $T$ such that there is a 2-sphere $\ob{S}$ in $M$ bounding a 3-ball in $M$ containing $T$. Give $T$ weight 2 and let $W$ be the 3-manifold such that there is an orientation preserving involution of $W$ whose quotient produces the orbifold $(M,T)$. Observe that $W$ is homeomorphic to the connected sum of $M$ with itself. This is depicted in Figure \ref{fig: exam 1}.

Let $S$ be the preimage of $\ob{S}$ in $W$. Since $\ob{S} \cap T = \nil$, the surface $S$ is the union of two spheres and $W|_S$ is the disjoint union of two lens spaces (each homeomorphic to $M$) and a copy of $S^3$. Thus, $|G|\nu/12 - |G|\mu/2 + |S| + 1 = 2$. Notice that $(M,T)$ has an orbifold Heegaard surface $\ob{H}$ with $\X(\ob{H}) = 1$. The surface $\ob{H}$ is a torus intersecting $T$ twice. The preimage of $\ob{H}$ in $W$ is an invariant Heegaard surface $H \subset W$ of genus 2. Thus, in this case, our inequality is sharp.
\end{example}

\begin{figure}[ht!]
\labellist
\small\hair 2pt
\pinlabel $M$  at 287 95
\pinlabel $M$  at 103 321
\pinlabel $M$  at 277 321
\pinlabel $T$ [tl] at 197 36
\pinlabel $\ob{S}$ [br] at 132 90
\pinlabel {$S$} [tr] at 56 258
\pinlabel {$S$} [tl] at 332 258
\pinlabel {$S^2 \times I$} [t] at 190 347
\pinlabel {$2:1$} [l] at 192 189
\pinlabel {$B^3$} at 186 101
\pinlabel {$\ob{H}$} [l] at 311 55
\pinlabel {$H$} [l] at 385 286
\endlabellist
\centering
\includegraphics[scale = 0.5]{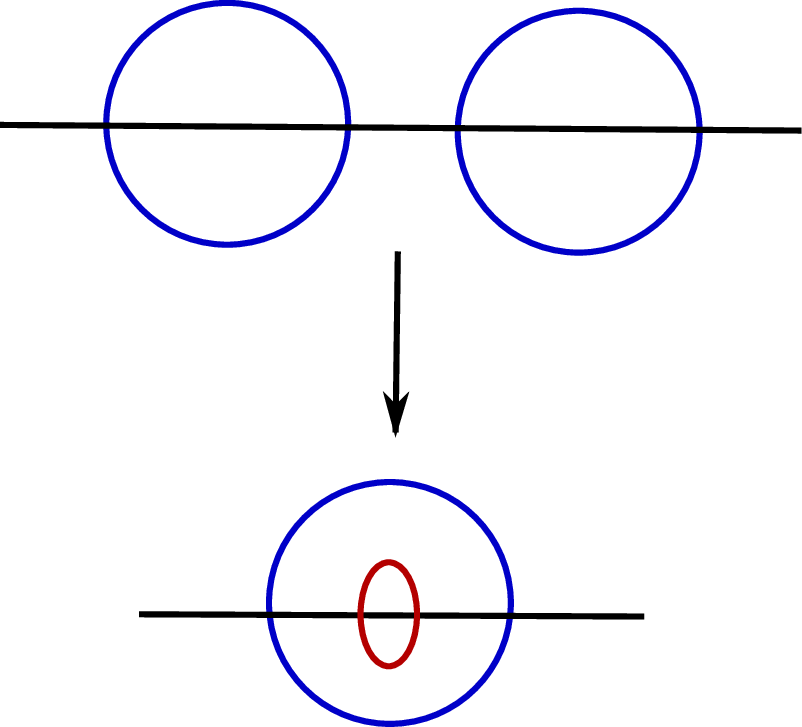}
\caption{The covering of the orbifold $(M,T)$ by $(W, \nil)$ in Example \ref{ex 1}. The manifold $M$ is a lens space and $W = M \times M$. The surface $S$ is the union of two disjoint spheres; it is a double cover of the unpunctured sphere $\ob{S}$. The knot $T$ is an unknot contained in 3-ball bounded by the sphere $S$. The horizontal lines represent bridge surfaces for $(M,T)$ and $(W, \nil)$, with $\ob{H}$ being a twice punctured torus and $H$ being a genus 2 surface.}
\label{fig: exam 1}
\end{figure}

\begin{example}\label{exam 2}
Let $M_1 = S^3$ and let $T_1$ be the spatial graph constructed as follows, and depicted in Figure \ref{fig: exam 2}. Choose a 2-bridge knot $K$ and and attach both an upper and lower tunnel to obtain $T_1$. The graph $T_1$ has four edges that intersect a bridge sphere $H_1$ and two edges (the upper and lower tunnels) that are disjoint from $H_1$. Give one of the edges intersecting $H_1$ a weight of 3 and give all the other edges weight 2. The subgraph $T_0$ of $T_1$ with edges of weight 2 is a trivial $\theta$-curve in $S^3$. 
The double-branched cover of $S^3$ over a cycle in $T_0$ is again $S^3$ and the third edge of $T_0$ lifts to an unknot. Taking another double-branched cover over that unknot again produces $S^3$. Thus, $S^3$ is a 4-fold orbifold cover over $(M_1, T_0)$. The edge of weight 3 in $T$ lifts to a knot $\kappa$ in the cover. The sphere $H_1$ lifts to a bridge sphere for $(S^3, \kappa)$ such that $|H_1 \cap \kappa| = 4$. In particular, $\kappa$ is either the unknot or a 2-bridge knot. Since $K$ is knotted, $\kappa$ is a 2-bridge knot. 

Let $W_1$ be the 3-fold branched cover of $S^3$ over $\kappa$.  Thus, $(M_1, T_1)$ is the quotient of $W_1$ by a group of diffeomorphisms of order 12 (isomorphic to the product of two groups of order 2 and one of order 3) . As long as $\kappa$ is not a torus knot or the figure-eight knot, $W_1$ is hyperbolic \cite[Chapter 1]{Cooper}.  Let $(M,T)$ be the distant sum of $(M_1, T_1)$ with $(M_2, \nil)$, a lens space having empty singular set. Let $\ob{S}$ be the summing sphere. The action of $G$ on $W_1$ extends to an action of (a group isomorphic to) $G$ on the connected sum $W$ of $W_1$ with 12 copies of $M_2$. The sphere $\ob{S}$ lifts to 12 copies of $S^2$. The right hand side of the inequality in Theorem \ref{counting lower bd} is then 14. The orbifold $(M, T)$ admits an orbifold Heegaard surface with orbifold characteristic 13/6. (It is a torus having three punctures of weight 2 and one of weight 3.) This lifts to an invariant Heegaard surface of $W$ having genus 14. So our lower bound is sharp in this case as well.
\end{example}

\begin{figure}[ht!]
\labellist
\small\hair 2pt
\pinlabel $3$ [l] at 239 179
\pinlabel $2$ [l] at 145 179
\pinlabel $2$ [r] at 92 179
\pinlabel $2$ [l] at 22 179
\pinlabel $2$ [t] at 109 245
\pinlabel $2$ [b] at 130 15
\pinlabel $H_1$ [l] at 280 162
\endlabellist
\centering
\includegraphics[scale = 0.5]{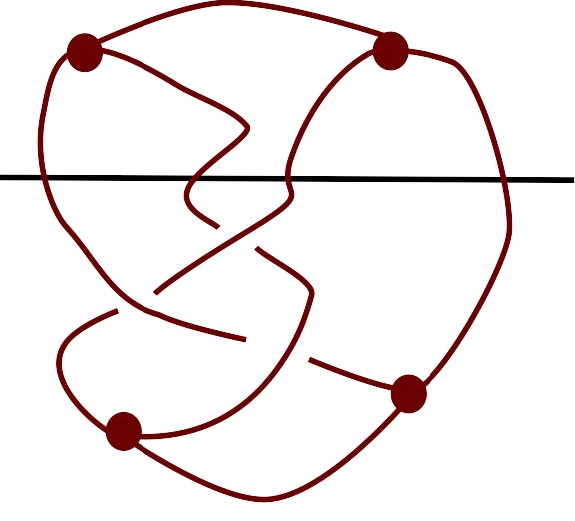}
\caption{The singular set $T_1$ and the bridge sphere $H_1$ for the orbifold $(M_1, T_1)$ in Example \ref{exam 2}}
\label{fig: exam 2}
\end{figure}

\subsection{Comparatively small factors}

A 3-orbifold $(M,T)$ is \defn{comparatively small} if each c-essential surface $F \subset (M, T)$ has $\X(F) > \X(M,T)$.

\begin{lemma}\label{lem: comp small}
Suppose that $(M,T)$ is a comparatively small nice orbifold. Then $\netX(M,T) = \X(M,T)$.
\end{lemma}
\begin{proof}
By definition, $\X(M,T) \geq \netX(M,T)$. We now show $\X(M,T) \leq \netX(M,T)$. Let $\mc{H} \in \H(M,T)$ be locally thin, with $\netX(\mc{H}) = \netX(M,T)$. We claim $\mc{H}^- = \nil$. 

Suppose, for a contradiction, that $\mc{H}^- \neq \nil$. The dual digraph to $\mc{H}$ has at least one source and at least one sink. Each such source and sink corresponds to a $(C, T_C) \cpt (M, T) \setminus \mc{H}$ with $\boundary_- C \subset \boundary M$. Let $(C, T_C)$ be one such and let $(D, T_D) \cpt (M, T) \setminus \mc{H}$ be distinct from $(C, T_C)$ but have $\boundary_+ D =\boundary_+ C$. Then $\mc{H}^- \cap \boundary_- D \neq \nil$; let $F$ be a component. Since $\mc{H}$ is locally thin, $F$ is c-essential in $(M, T)$ (Theorem \ref{locally thin}). Since $(M, T)$ is comparatively small, $\X(F) > \X(M, T)$. By Lemma \ref{low N},
\[
N(C, T_C) + \X(\boundary_- C) = \X(\boundary_+ C) = \X(\boundary_+ D) \geq \X(F) > \X(M, T) \geq \netX(M, T). 
\]
Since there are at least two such $(C, T_C)$, by \eqref{Fundamental} and the niceness of $(M, T)$, we conclude $\netX(M,T) > \netX(M, T)$, a contradiction. Thus, $\mc{H}^-= \nil$. Consequently, $\mc{H}$ is an orbifold Heegaard surface for $(M,T)$ and so $\netX(\mc{H}) \geq \X(M,T)$.
\end{proof}

\begin{theorem}\label{comp small}
Suppose that $(M,T)$ is a nice 3-orbifold. Assume also that for some (hence, any) efficient system of summing spheres $S$ each component of $(M,T)|_S$ is comparatively small. Then
\[
\X(M,T) \geq \netX(M,T) \geq \X((M,T)|_S)- \X(S).
\]
\end{theorem}
\begin{proof}
If $(M,T)$ is orbifold-irreducible, the result is vacuously true. Suppose that $(M,T)$ is orbifold-reducible. By Theorem \ref{factor uniqueness}, each component of $(M,T)|_S$ is comparatively small for any efficient system of summing spheres $S$. By definition, $\X(M,T) \geq \netX(M,T)$. By Theorem \ref{additivity}, $\netX(M,T) \geq \netX((M,T)|_S) - \X(S)$ for an efficent system of summing spheres $S$. By Lemma \ref{lem: comp small}, $\netX((M,T)|_S) = \X((M,T)|_S)$, and the result follows.
\end{proof}

\begin{theorem}\label{thm: equiv compare small}
Suppose that when $S \subset W$ is an invariant system of summing spheres, then every component of $W|_S$ is equivariantly comparatively small. Then:
\[
\g(W;G) \geq \g(W|_S;G).
\]
\end{theorem}
\begin{proof}
Without loss of generality, we may assume that $W$ is reducible. Let $(M,T)$ be the quotient orbifold. Observe that it is nice and orbifold-reducible. Let $\ob{S}$ be an efficient system of summing spheres for $(M,T)$. We claim that each component of $(M,T)|_S$ is comparatively small. To see this, suppose that $(M_i, T_i) \cpt (M,T)|_S$ and that $\ob{F} \subset (M_i, T_i)$ is a c-essential surface. Let $W_i$ be a component of the preimage of $(M_i, T_i)$. Since $(M_i, T_i)$ is orbifold-irreducible, $\ob{F}$ is not a sphere. Thus, no component of the preimage of $\ob{F}$ to $W$ is a sphere. Let $F$ be the preimage of $\ob{F}$ in $W|_S$, with $F_i = F \cap W_i$. Let $G_i \subset G$ be the stabilizer of $W_i$. We have $\X(\ob{F}) =\X(F_i)/|G_i|$. 

By the equivariant loop theorem \cite{MY-EDL}, $F$ is incompressible. Suppose that some $F' \cpt F$ is $\boundary$-parallel in $W$. Let $W' \cpt W \setminus F'$ be homeomorphic to $F' \times I$. We may assume that $F'$ was chosen so that the interior of $W'$ is disjoint from $F$. The image of $W'$ in $(M,T)$ is the quotient of a trivial product compressionbody (namely $W'$) by a finite group of diffeomorphisms (the stabilizer of $F'$). By Lemma \ref{quotient info}, it a product trivial compressionbody. Thus, $\ob{F}$ bounds a trivial product compressionbody with a component of $\boundary W$ and so $\ob{F}$ is not c-essential, a contradiction. Thus, $F$ is essential in $W|_S$.  Since $W_i$ is equivariantly comparatively small, $\X(F_i) > \X(W_i)$. Thus,
\[
\X(\ob{F}) = \X(F_i)/|G_i| > \X(W_i)/|G_i| = \X(M_i, T_i).
\]
Thus, each component of $(M,T)|_{\ob{S}}$ is comparatively small. Our result follows from Theorem \ref{comp small} after multiplying by $|G|$.
\end{proof}

\section{Upper Bounds}\label{sec: upper}
In this section we use multiple orbifold Heegaard surfaces to construct upper bounds on the equivariant Heegaard genus of $W$ and explain why the bound is sharp. Throughout this section, we use the inverse operation to ``undoing a removable arc'' mentioned earlier.

\begin{definition}\label{create removable}
Suppose that $(M,T)$ is a nice orbifold and that $\mc{H} \in \H(M,T)$. Suppose also that an edge $\alpha$ of $T \setminus \mc{H}$ is a ghost arc contained in $(C, T_C) \cpt (M,T)\setminus \mc{H}$. Choose a cut disc or semi-cut disc $D \subset (C, T_C)$ intersecting $\alpha$ exactly once. Choose an arc $\kappa$ in $D$ from $\boundary D$ to $D \cap \alpha$. In a neighborhood of $D$ in $M$, isotope $\alpha$ by an isotopy following $\kappa$, so that the interior of $\alpha$ is pushed across $\boundary_+ C$, as in Figure \ref{fig: create remov}. Dually, we may isotope $\mc{H}$. This converts $\mc{H}$ into a new $\mc{J} \in \H(M,T)$ such that
\[
\netX(\mc{J}) = \netX(\mc{H}) + 2(1 - 1/\weight(\alpha)).
\]
The ghost arc $\alpha$ is converted into the union of two vertical arcs and a bridge arc on the opposite side of $\boundary_+ C$ from the vertical arcs. We call this move \defn{creating a removable arc from $\alpha$}.
\end{definition}

\begin{figure}[ht!]
\centering
\includegraphics[scale=0.8]{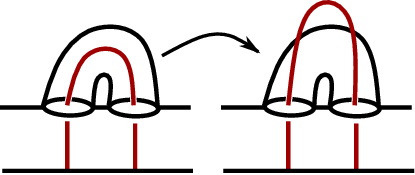}
\caption{Creating a removable arc}
\label{fig: create remov}
\end{figure}

In 3-manifold theory, generalized Heegaard splittings may be amalgamated to create Heegaard splittings \cite{Schultens}. In our situation, it is not always possible to amalgamate multiple orbifold Heegaard surfaces to create Heegaard surfaces. Ghost arcs are the obstruction and creating removable arcs allows us to amalgamate. 

\begin{definition}
Suppose that $(M,T)$ is an orbifold and that $\mc{H} \in \H(M,T)$. Suppose that $(M_i, T_i) \cpt (M,T)\setminus \mc{H}^-$ for $i = 1,2$ and that $F = \boundary M_1 \cap \boundary M_2 \neq \nil$. Let $H_i = \mc{H}^+ \cap M_i$. We say that $H_1$ and $H_2$ are \defn{amalgable} if whenever $e_i \subset T_i \setminus \mc{H}$ for $i = 1,2$ are edges sharing an endpoint, then at least one is not a ghost arc.
\end{definition}

\begin{proposition}[Amalgamation]\label{amalg}
If $H_1$ and $H_2$ are amalgable, then there exists $\mc{J} \in \H(M,T)$ and $H \cpt \mc{J}^+$ such that $\mc{J} \setminus H = \mc{H} \setminus (H_1 \cup H_2 \cup F)$ and $\netX(\mc{J}) = \netX(\mc{H})$. Furthermore, each ghost arc of $T \setminus \mc{J}$ contains at least one ghost arc of $T \setminus \mc{H}$.
\end{proposition}
\begin{proof}
Our proof is similar to that of \cite[Theorem 4.1]{Saito}. Let $(C, T_C), (D, T_D) \cpt (M,T) \setminus \mc{H}$ be the components with $H_1 = \boundary_+ C$ and $H_2 = \boundary_+ D$. Choose complete collections of sc-discs $\Delta_C \subset (C, T_C)$ and $\Delta_D \subset (D, T_D)$. We choose these discs somewhat carefully. They must have the properties that the only edges of $T_C$ and $T_D$ intersecting $\Delta_C$ and $\Delta_D$ respectively are ghost arcs or core loops and that each ghost arc and core loop intersects exactly one such disc. It is not difficult to see that $\Delta_C$ and $\Delta_D$ can be chosen to satisfy these conditions. For simplicity in the discussion, assume that $F$ is connected. If it is not, the proof goes through with only minor changes in wording.

Boundary reduce $(C, T_C)$ and $(D, T_D)$ using $\Delta_C$ and $\Delta_D$  to obtain $(C', T'_C)$ and $(D', T'_D)$ respectively. One component of each is a trivial product compressionbody having $F$ as a boundary component. Observe that these components are disjoint from any core loops in $T_C$ or $T_D$. Let $X$ be their union with $T_X = T \cap X$. Notice that $X$ is homeomorphic to $F \times I$. Let $\boundary_C X$ be the component of $\boundary X$ intersecting $\boundary_+ C$ and $\boundary_D X$ the component intersecting $\boundary_+ D$. Each component of $T_X$ is a vertical arc. Consider the scars $\delta_C$ on $\boundary_C X$ resulting from the $\boundary$-reduction along $\Delta_C$. Some of them are unpunctured discs and others are punctured discs. Each punctured disc is incident to an arc of $T_X$ whose other endpoint may lie on $F$, but does not lie in the scars $\delta_D \subset \boundary_D X$ resulting from the $\boundary$-reduction along $\Delta_D$. Extend $ B = \delta_C \cap \boundary_C X$ vertically through $X$ via a solid tubes $B \times I$. The \defn{frontier} of the tubes is $\boundary B \times I$. Since no component of $T_X$ is incident to both $\delta_C$ and $\delta_D$, we may shrink the tubes of $B$ containing an arc of $T_X$ and isotope the ends on $\boundary_D X$ of the other tubes so that the tubes are disjoint from $\delta_D$. 

Delete the discs that are the ends of the tubes from $H_2$ and attach the frontiers of the tubes. When we compressed $H_1$ along $\Delta_C$, we used a regular neighborhood that can be parameterized as $\Delta_C \times I$. Reattach $\boundary \Delta_C \times I$ to the ends on $\boundary_C X$ of the frontiers of the tubes. Call the surface we created $H$. 

We claim that $H$ is a vp-bridge surface for $(M', T') = (M_1, T_1) \cup_F (M_2, T_2)$. Let $(D', T'_D)$ be the other vp-compressionbody of $(M,T) \setminus \mc{H}$ with $\boundary_+ D' = H_2$ and $(C', T'_C)$ the other vp-compressionbody with $\boundary_+ C' = H_1$. Let $(U, T_U)$ be the component of $(M', T') \setminus H$ containing $(D', T'_D)$. Let $(V, T_V)$ be the other component. Observe that the union of $\Delta_C$ with a complete set of sc-discs for $(D', T'_D)$ is a complete set of sc-discs for $(U, T_U)$. Thus, $(U, T_U)$ is a vp-compressionbody. Let $\Delta'_C$ be a complete set of sc-discs for $(C', T'_C)$. We can extend $\boundary \Delta'_C \cap \boundary_C X$ through $X$ to lie on $\boundary_D X$. As before, we can ensure they miss the scars of $\Delta_D$. The union of these discs with $\Delta_D$ is then a complete set of sc-discs for $(V, T_V)$, showing that it is also a vp-compressionbody. Thus, $H$ is a vp-bridge surface for $(M', T')$.

Let $\mc{J}$ be as in the statement of the proposition. The dual digraph to $\mc{H}$ contains a connected subgraph $\alpha$ with edges that corresponding to the surfaces $H_1$, $F$, and $H_2$. It has a single source and a single sink. Give $H$ the transverse orientation inherited from $H_1$ and $H_2$. The dual digraph to $\mc{J}$ is obtained from that of $\mc{H}$ by replacing $\alpha$ with a single edge. Since the dual digraph to $\mc{H}$ was acyclic, so is the dual digraph to $\mc{J}$. Thus, $\mc{J} \in \H(M,T)$. The computation $\netX(\mc{J}) = \netX(\mc{H})$ is easily verified by compressing $H$ along the discs $\Delta_C$ to recover $H_2$. 

Finally, suppose that $\alpha$ is a ghost arc of $T \setminus \mc{J}$. Traversing $\alpha$ we start at a thin surface in $\mc{J}^-$ or vertex of $T$ and end on a thin surface of $\mc{J}^-$ or vertex of $T$. If during the traversal, we never traverse an arc of $T_X$, then $\alpha$ is a ghost arc of $T \setminus \mc{H}$. If we do traverse an arc of $T_X$, then $\alpha$ must have one endpoint in $\delta_C \cup \delta_D$, as otherwise it wouldn't be a ghost arc of $T \setminus \mc{J}$. But this implies that it contains a ghost arc of $T \setminus \mc{H}$.
\end{proof}

\subsection{Examples of super additivity}\label{sec: super}

\begin{theorem}\label{super-add}
For $k \geq 2$, there exist reducible $W$ having a finite cyclic group of diffeomorphisms $G$ of order $k$, and an invariant essential sphere $S$ dividing $W$ into two irreducible manifolds, such that
\[
\g(W;G)  = \g(W|_S; G)  + k - 1.
\]
Furthermore, $\g(W;G)$ can be arbitrarily high when $G$ is a cyclic group of fixed order.
\end{theorem}

Associated to each Heegaard surface $H$ of genus at least 2 of a compact 3--manifold $X$, is a nonnegative integer invariant $d(H)$, called \defn{Hempel distance} \cite{Hempel}. When $d(H) = 0$, the Heegaard surface $H$ is reducible. When $d(H) = 1$, the Heegaard surface $H$ can be untelescoped to a generalized Heegaard surface with multiple components. Essential surfaces or strongly irreducible Heegaard surfaces can often be used to provide upper bounds on Hempel distance. Furthermore, if every thick component of a multiple Heegaard surface has high enough Hempel distance, the generalized Heegaard surface is locally thin. We apply this philosophy in our context.

By \cite{MMS}, for each $t \in \N$ and $N \in \N$, there exists a knot $K \subset S^3$ such that $S^3 \setminus K$ admits a Heegaard surface of genus $t + 1$ and Hempel distance $d(H) \geq N$. Fix $t, N, k \in \N$. For $i = 1,2$, choose knots $K_i$ in $S^3$ such that each has a Heegaard surface $H_i$ of its exterior of genus $t + 1$ (and $\X(H_i) = 2t$) and Hempel distance at least $\zeta = 2(4t + N+  2(1-1/k) + 2)/(1 - 1/k) + 3$. We use the following, drawn from work of Scharlemann, Scharlemann--Tomova and Tomova.  The statement for $F$ is due to Scharlemann \cite[Theorem 3.1]{Sch-Proximity} and is a generalization of \cite{Hartshorn}. The case when $J$ is disjoint from $K_i$ is the main result of \cite{SchTo} and when $J$ intersects $K_i$, it is the main result of \cite{To-PJM}.  We rephrase their results, using our terminology.

\begin{theorem}[Scharlemann, Scharlemann-Tomova, Tomova]\label{SchTo}
If $F \subset (S^3, K_i)$ is an essential connected surface, then $-\chi(F) + |F \cap K_i| >  \zeta - 3$. If $J \in \H(S^3, K_i)$ is connected, then either $-\chi(J) + |J \cap T| > \zeta - 3$ or $J \to H_i$.
\end{theorem}

Let $K = K_1 \# K_2$. Consider $(S^3, K)$ as an orbifold where $K$ is given weight $k \in \WeightSet$. Note that $(S^3, K)$ has a multiple orbifold Heegaard surface $\mc{H}$ with $\mc{H}^+ = H_1 \cup H_2$ and $S = \mc{H}^-$ the twice-punctured summing sphere realizing $K$ as a connected sum of $K_1$ and $K_2$. (We also need to give $\mc{H}$ one of the two orientations making it an oriented multiple vp-bridge surface.)  For the record, we have $\netX(\mc{H}) = 4t + 2/k$. By the definitions of Hempel distance and the partial order $\to$ (both of which we have omitted), $\mc{H}$ is locally thin.

\begin{lemma}\label{verify min}
We have $\netX(S^3, K) = \netX(\mc{H})$. Furthermore, if $\mc{J} \in \H(M,T)$ is locally thin and has $\netX(\mc{J}) \leq  \netX(S^3, K) + N + 2(1 - 1/k)$, then up to isotopy and orientation reversal, $\mc{H}$ can be obtained from $\mc{J}$ by deleting pairs of twice-punctured spheres from $\mc{J}^-$ and $\mc{J}^+$. In particular, $\netX(\mc{J}) = \netX(\mc{H})$.
\end{lemma}
\begin{proof}
Let $\mc{J} \in \H(S^3, K)$ be locally thin and have $\netX(\mc{J}) \leq  \netX(S^3, K) + N + 2(1 - 1/k)$. Since $(S^3, K)$ is orbifold-reducible, $\mc{J}^-$ contains an efficient system of summing spheres for $(S^3, K)$. By Theorem \ref{SchTo} (or, indeed, by \cite{Hempel}), both $K_1$ and $K_2$ are prime knots. Consequently, there is a unique essential summing sphere for $(S^3, K)$, up to isotopy.  Isotope $\mc{J}$ so that $S \cpt \mc{J}^-$. Suppose that there exists $F \cpt \mc{J}^- \setminus S$. Choose $F$ so that it bounds a 3-submanifold $X \subset S^3$ with interior disjoint from $\mc{J}^-$ (i.e. $F$ is outermost). Let $J = \mc{J}^+ \cap X$ and let $(C, T_C) \cpt (S^3, K) \setminus \mc{J}$ be the orbifold compressionbody with $J = \boundary_+ C$ and $\boundary_- C = \nil$. Without loss of generality, we may assume that $F$ is on the same side of $S$ as $K_1$.  If $F$ is essential in $(M,T)\setminus S$, then by Theorem \ref{SchTo}, 
\[\begin{array}{rcl}
2(4t+N +  2(1-1/k) + 2) &=& (1 - 1/k)(\zeta - 3) \\
&<& (1 - 1/k)(-\chi(F) + |F \cap K_1|)  \\
&\leq& -\chi(F) + (1 - 1/k)|F \cap K_1| \\
&=& \X(F).
\end{array}
\]

By Lemma \ref{low N}, $N(C, T_C) = \X(H) \geq \X(F)$ and so, by \eqref{Fundamental}, 
\[\begin{array}{rcl}
\netX(S^3, K) + N + 2(1 - 1/k) &\geq& \netX(\mc{J}) \\
&>& 4t + N + 2(1 - 1/k) + 2 \\
&=& \netX(\mc{H}) + N + 2(1 - 1/k) - 2/k + 2\\
&\geq& \netX(S^3, K) + N + 4(1- 1/k) \\
\end{array}.
\]
This is a contradiction. Thus, $F$ must be inessential in $(S^3, K)\setminus S$. Since $F$ is c-essential in $(S^3, K)$, it must be isotopic to $S$. Perform the isotopy to make $F$ coincide with $S$. After the isotopy, $X$ is the side of $S=F$ containing $K_1$. By Theorem \ref{SchTo}, either  $-\chi(J) + |J \cap T| > \zeta - 3$ or $J \to H_i$. In the former case, we have:
\[
N(C, T_C) = \X(J) = -\chi(J) + (1 - 1/k)|J \cap T| > (1 - 1/k)(\zeta -3) = 2(4t + N+  2(1-1/k) + 2).
\]
The same arithmetic as before establishes a contradiction. Thus, $J \to H_i$. However, $\mc{J}$ is locally thin, so in fact, $J$ is isotopic to $H_i$ (ignoring orientations). 

We have shown that each outermost $F \cpt \mc{J}^-$ is isotopic to $S$. Furthermore, if $\mc{J}^-$ is connected, then $\mc{J}$ is isotopic to $\mc{H}$, ignoring orientations. Suppose that $\mc{J}^-$ is disconnected. Since every surface in $S^3$ separates, there are at least two outermost $F_1, F_2 \cpt \mc{J}^-$. They cobound $(Y, T_Y) \subset (S^3, K)$ that is a product compressionbody homeomorphic to $(S^2 \times I, \{p_1, p_2\} \times I)$, where $p_1, p_2 \in S^2$.  Since each component of $\mc{J}^-$ is c-essential, each component of $\mc{J}^-$ must be parallel to each of $F_1$ and $F_2$ in $(Y, T_Y)$. Suppose that $F_i, F_{i+1} \cpt \mc{J}^-$ cobound a submanifold $(Y', T'_Y) \cpt (Y, T_Y)\setminus \mc{J}^-$. Note that $(Y', T'_Y)$ is homeomorphic to $(Y, T_Y)$. Let $J' = \mc{J}^+ \cap Y'$. Since $\mc{J}$ is locally thin, $J'$ must be a twice-punctured sphere bounding a trivial ball compressionbody $B(J')$ to one side. We conclude that
\[
\netX(\mc{J}) = \X(H_1) + \X(H_2) + \frac{-2}{k}(|\mc{J}^-|-1) - \frac{-2}{k}|\mc{J}^-| = \X(H_1) + \X(H_2) + \frac{2}{k} = \netX(\mc{H}).
\]
Thus, $\netX(\mc{H}) = \netX(S^3, K)$. Furthermore, after deleting all the components of $\mc{J} \cap Y$ except $F_1$, we obtain a multiple vp-bridge surface isotopic to $\mc{H}$, ignoring orientations.
\end{proof}

We will also need the following:
\begin{definition}
Suppose that $\mc{H} \in \H(M,T)$. The \defn{net geometric intersection number} of $\mc{H}$ is $\netp(\mc{H}) = |\mc{H}^+ \cap T| - |\mc{H}^- \cap T|$.
\end{definition}

We can now calculate the Heegaard characteristic of $(S^3, K)$.
\begin{lemma}\label{Heeg gen}
$\X(S^3, K) = 4t + 2$
\end{lemma}
\begin{proof}
There are four orbifold compressionbody components of $(S^3, K)\setminus \mc{H}$. The two not adjacent to $S = \mc{H}^-$ are disjoint from $K$ and each of the other two contain a single ghost arc, both of whose endpoints are on $S$. Choose one of them and use it to create a removable arc. Call the new multiple orbifold Heegaard surface $\mc{H}'$ with two thick surfaces and thin surface $S$. Note that the thick surfaces are amalgable. Amalgamate them, converting $\mc{H}'$ into a connected $H' \in \H(M,T)$ with
\[
\netX(H') = \netX(\mc{H}') = \netX(M,T) + 2(1 - 1/k) = 4t + 2.
\]

Let $J$ be an orbifold Heegaard surface for $(S^3, K)$ such that $\X(J) = \X(S^3, K)$. There is a locally thin $\mc{J} \in \H(S^3, K)$ such that $J \to \mc{J}$. By Lemma \ref{verify min},
\[
\netX(M,T) + 2(1-1/k)= \X(H') \geq \X(J) \geq \netX(\mc{J}) \geq \netX(\mc{H}) = \netX(M,T).
\]

By Lemma \ref{verify min}, $\netX(\mc{J}) = \netX(\mc{H})$ and after deleting pairs of twice-punctured spheres from $\mc{J}^-$ and $\mc{J}^+$, $\mc{J}$ is isotopic to $\mc{H}$ (ignoring orientations). As in Lemma \ref{move facts}, the thinning moves that create $\mc{J}$ from $J$ potentially consist of four types of moves. Since $S^3$ is closed and $K$ is a knot,  of the moves listed in Type (I), we never need to perform a $\boundary$-destabilization, meridional $\boundary$-destabilization or meridional ghost $\boundary$-destabilization. Performing a ghost $\boundary$-destabilization, involves compression along a separating compressing disc and the discarding of a torus boundary component. (That torus is isotopic to the frontier of a regular neighborhood of $K$.) Such a move decreases negative orbifold Euler characteristic of a thick surface by $2 \geq 2(1 - 1/k)$. Destabilization also decreases the negative orbifold Euler characteristic of a thick surface by 2. Meridional destabilization decreases it by $2/k$. The moves in Type (II) decrease it by $2(1-1/k)$. Consolidation and untelescoping leave $\netX$ unchanged. 

We see, therefore, that in the thinning sequence producing $\mc{J}$ from $J$, there can be at most one move that is a ghost $\boundary$-destabilization, destabilization, unperturbation, or undoing a removable arc. All other moves are either meridional destabilization, untelescoping, or consolidation. Ghost $\boundary$-stabilization, destabilization, meridional destabilization, untelescoping, and consolidation do not decrease $\netp$. Unperturbing and undoing a removable arc decrease $\netp$ by $2$. 

Observe that $\netp(\mc{J}) = \netp(\mc{H}) = -2$. Since $J$ is connected, $K \setminus J$ contains no ghost arcs. Thus, $\netp(J) \geq 0$. Thus, at least one unperturbing or undoing a removable arc are required in the thinning sequence producing $\mc{J}$ from $J$. We have already seen that there is at most one, so there must be exactly one and we cannot have any meridional destabilizations. We conclude that the thinning sequence producing $\mc{J}$ from $J$ consists of exactly one unperturbation or undoing a removable arc and some number of untelescopings and consolidations. We conclude that $\X(J) = \X(H')$. Thus, $\X(S^3, K) = \X(H') = 4t + 2$. 
\end{proof}

The first examples of pairs of knots producing super-additivity of tunnel number were given in \cites{MSY, MR}. Setting $k = \infty$ produces other examples, using a similar method to \cites{YL09, LYL}.

\begin{proof}[Proof of Theorem \ref{super-add}]
Fix $t, k \in \N$ with $k \geq 2$. For each $N \in \N$, construct $K_1, K_2$ as above. By Lemma \ref{Heeg gen}, $\X(S^3, K) = 4t + 2$. Let $W$ be the $k$-fold cyclic branched cover over $K$, with $G$ the deck group. Then, by Lemma \ref{quotient info},  $\X(W;G) = 4tk + 2k$. The manifold $W$ is the connected sum of $W_1$ and $W_2$ which are the $k$-fold branched covers over $K_1$ and $K_2$ respectively. By Lemma \ref{verify min}, $\X(W_i) = 2t$.  Let $S$ be the lift of a summing sphere for $(S^3, K)$. Observe that $S$ is efficient as $K_1$ and $K_2$ are prime. Thus $\X(W|_S;G) = 4tk$ and $\X(W;G) = \X(W|_S;G) + 2k$. Converting to genus, we have
\[
\g(W;G) = \g(W|_S;G) + k - 1.
\]
\end{proof}

\subsection{A general upper bound}

In this section, we adapt our example to produce a general upper bound for equivariant Heegaard genus of composite manifolds. As usual, we start by considering orbifolds.

\begin{theorem}\label{orb upper}
Suppose that $(M,T)$ is orbifold-reducible and that $S$ is a system of summing spheres, with each component of $S$ separating. Then
\[
\X(M,T) \leq \X((M,T)|_S) - \X(S)(1 - c) + 2c|S|
\]
where $c = 1$ if $T$ has no vertices and $c = 2$ if it does.
\end{theorem}

\begin{proof}
Let $(M_i, T_i)$ for $i = 1, \hdots, n$ be the components of $(M,T)|_S$. Choose an orbifold Heegaard surface $H_i \subset (M_i, T_i)$ such that $\X(M_i, T_i) = \X(H_i)$. Assign transverse orientations chosen so that if we set $\mc{H}^+ = \bigcup H_i$ and $\mc{H}^- = S$, then $\mc{H} = \mc{H}^+ \cup \mc{H}^- \in \H(M,T)$. This is possible since the dual graph to $S$ is a tree. Thus,
\[
\netX(\mc{H}) = \sum \X(H_i) - \X(S). 
\]

Consider a point $p \in T \cap S$ such that both edges of $T \setminus \mc{H}$ incident to $p$ are ghost arcs. Let $S_0 \cpt S$ contain $p$. As the two ghost arcs lie in the same edge of $T$, they have the same weight $\weight(p)$. Let $\alpha$ be the one lying on the same side of $S$ as its normal vector and suppose $\alpha \subset (M_i, T_i)$. Perform the isotopy of Definition \ref{create removable} to create a removable arc from $\alpha$. The isotopy converts $\alpha$ into the union of a bridge arc and two vertical arcs, with the bridge arc on the opposite side of $H_i$ from the two vertical arcs. Dually, we may isotope $H_i$. After the isotopy, $\X(H_i)$ has increased by $2(1 - 1/\weight(p))$. Do this for each such ghost arc in $(M_i, T_i)$ incident at a puncture of $S_0$ to a ghost arc on the opposite side of $S$. The most we increase $\X(H_i)$ by is $2\sum (1- 1/\weight(p))$, where the sum is over the punctures of $S_0$. This is equal to $4 + 2\X(S_0)$. Doing the same thing for each component of $S$ increases $\netX(\mc{H})$ by at most $4|S| + 2\X(S)$. However, notice that if a ghost arc has both its endpoints on $S$ (rather than on $S$ and a vertex of $T$) then at worst we only need to perform our move once for each point of $|S \cap T|$, rather than twice. Such will be the case if $T$ has no vertices, for example. In that case, we increase $\X(\mc{H})$ by at most $\sum_p (1 - 1/p)$ where the sum is over all punctures $p$ of $S$. In such a case, we increase $\X(\mc{H})$ by at most $2|S| + \X(S)$. We then have a new $\mc{J} \in \H(M,T)$ such that
\[
\netX(\mc{H}) \leq \netX(\mc{J}) \leq \netX(\mc{H}) + D,
\]
where $D = 2|S| + \X(S)$ if $T$ is a link and $D = 4|S| + 2\X(S)$ otherwise. Notice that no two ghost arcs of $T \setminus \mc{J}$ are incident to the same point of $T \cap S$. 

Amalgamation does not create additional ghost arcs, so by Proposition \ref{amalg} we may amalgamate the thick surfaces of $\mc{J}$ two at a time to eventually obtain a connected $J \in \H(M,T)$ such that $\netX(\mc{J}) = \X(J)$. Thus,
\[
\X((M,T)|_S) - \X(S) + D \geq \netX(\mc{J}) = \X(J) \geq \X(M,T). 
\]
This can be rearranged into the claimed inequality.
\end{proof}

\begin{theorem}\label{equiv-upper}
If every sphere in $W$ separates, then for any equivariant system of summing spheres $S \subset W$,
\[
\g(W;G) \leq \g(W|_S;G) + (c(|G|+1) - 2)\Big(\big|W|_S\big|-1\Big)
\]
where $c = 1$ if every point of $W$ has cyclic stabilizer and $c = 2$ otherwise.
\end{theorem}

\begin{proof}
If $W$ is irreducible, then we can take $S = \nil$ and $n = 1$ and the result is vacuously true. Suppose that $W$ is reducible. By Theorem \ref{est}, the quotient orbifold $(M,T)$ is orbifold-reducible. Recall that $T$ has no vertices if and only if every point of $W$ has cyclic stabilizer. Let $\ob{S}$ be an efficient system of summing spheres and $S$ its lift to $W$. Note that $-2(n-1) = -2|S| = |G|\X(S)$. By Lemma \ref{quotient info}, $\X(W;G) = |G|\X(M,T)$ and $\X(W|_S;G) = |G|\X((M,T)|_S)$. The result follows from Theorem \ref{orb upper} after converting the inequalities in the conclusion of that theorem to be in terms of genus.
\end{proof}

\begin{remark}
The examples of Theorem \ref{super-add} show that our upper bound is sharp when every point of $W$ has cyclic stabilizer. It is likely possible to adapt those examples to show that the inequality is sharp even when some points of $W$ do not have cyclic stabilizer.
\end{remark}

\appendix
\section{Efficient Factorizations}
In this section we sketch the proof that, even when there are nonseparating spherical suborbifolds, efficient factorizations exist and the factors are unique up to orbifold homeomorphism.

\begin{corollary-thin sum}
Suppose that $(M,T)$ is a nice orbifold and let $\mc{H} \in \H(M,T)$ be locally thin. Then $\mc{H}^-$ contains a system of summing spheres for $(M,T)$.
\end{corollary-thin sum}

\begin{proof}
If $(M,T)$ is not orbifold reducible, there is nothing to prove, so suppose that it is. Let $S \subset (M,T)$ be an essential summing sphere. Isotope it to intersect $\mc{H}^-$ in the minimal number of loops. Since $S$ has three or fewer punctures, each component of $S \cap \mc{H}^-$ is inessential in $S$. Let $\xi$ be an innermost such loop on $S$, with $D \subset S$ the unpunctured or once-punctured disc it bounds. Let $F \cpt \mc{H}^-$ contain $\xi$. Since $F$ is c-essential, the curve $\xi$ must be inessential on $F$. Let $E \subset F$ be the unpunctured or once-punctured disc it bounds. Since $(M,T)$ is nice, $|E \cap T| = |D \cap T|$. If both $E$ and $D$ are once-punctured, then since $(M,T)$ is nice, the weights of the punctures are the same. Observe that if $S$ intersects the interior of $E$, the intersection curves are also inessential in $F$. Let $\zeta \subset E$ be an innermost such component; possibly $\zeta = \xi$. Let $E' \subset E$ be the unpunctured or once-punctured disc it bounds in $F$ and $D' \subset S$ the unpunctured or once-punctured disc it bounds in $S$. Again, we have $|E' \cap T| = |D' \cap T|$ and if $E'$ and $D'$ are both once-punctured, then the punctures have the same weight.

Compress $S$ using $E'$ to arrive at $S'$. Note that $S'$ consists of two components, neither with more punctures then $S$. Although the total number of punctures may have gone up, it does so if and only if $E'$ and $D'$ are once-punctured. In which case, one component of $S'$ is a sphere with the same number of punctures and orbifold Euler characteristic as $S$ and the other is a twice-punctured sphere with both punctures of the same weight. If $E'$ was once-punctured, then neither component of $S'$ is an unpunctured sphere or a twice-punctured sphere bounding a trivial ball compressionbody, as $|S \cap \mc{H}^-|$ was minimized, up to isotopy. Similarly, if $E'$ was unpunctured, then neither component of $S'$ is an unpunctured sphere bounding a ball disjoint from $T$. Thus, some component of $S'$ is an essential summing sphere with no more punctures than $S$ and with $\X(S) \leq \X(S')$. Repeating the argument, we arrive at an essential summing sphere $S''$  for $(M,T)$ with $\X(S'') \leq \X(S)$, with $S''$ having no more punctures than $S$, and with $S'' \cap \mc{H}^- = \nil$. The argument of \cite[Theorem 7.2]{TT1} shows that we can further isotope $S''$ to be disjoint from $\mc{H}$. 

Let $(C, T_C) \cpt (M,T)\setminus \mc{H}$ be the component containing $S''$. Let $\Delta \subset (C, T_C)$ be a complete collection of sc-discs for $(C, T_C)$. Isotope $S''$ in $(C, T_C)$ to minimize $|\Delta \cap S''|$. The intersection consists of loops, since $S''$ is closed. An argument identical to that above shows that we can compress $S''$ to make it disjoint from $\Delta$. We may as well assume that $S''$ was disjoint from $\Delta$ to begin with. The sphere $S''$ then lies in a trivial compressionbody. Since it is c-essential, it is parallel to a component of $\boundary_- C$. All that is needed for the above argument to work is that each component of $\mc{H}^-$ is c-incompressible and each component of $\mc{H}^+$ is sc-strongly irreducible.

Thus, if $(M,T)$ contains an essential unpunctured sphere, so does $\mc{H}^-$. Let $S$ be the union of all unpunctured spheres in $\mc{H}^-$. Notice that $\mc{H}_S = \mc{H} \setminus S$ restricts to a multiple vp-bridge surface for each component of $(M,T)|_S$. It is easy to see that each component of $\mc{H}^-_S$ remains c-incompressible and each component of $\mc{H}^+_S$ is sc-strongly irreducible. Thus, $(M,T)|_S$ does not contain an essential unpunctured sphere. The previous argument then shows that if $(M,T)|_S$ contains an essential twice-punctured summing sphere, then $\mc{H}^-_S$ does as well. Let $Q \subset \mc{H}^-_S$ be the union of all the twice-punctured summing spheres. The same argument as in the unpunctured case shows that $(M,T)|_{S \cup Q}$ contains no essential unpunctured or twice-punctured summing spheres. Finally, repeat the analysis for thrice-punctured spheres to see that $\mc{H}^-$ contains a collection of summing spheres $P$ such that $(M,T)|_P$ contains no essential summing spheres. \end{proof}

\begin{theorem-factor uniqueness}[after Petronio, Hog-Angeloni--Matveev]
Suppose that $(M,T)$ is a nice orbifold that is orbifold-reducible. Then there exists an efficient system of summing spheres $S \subset (M,T)$; indeed, any system of summing spheres contains an efficient subset. Furthermore, any two such systems  $S$, $S'$ are orbifold-homeomorphic, as are $(M,T)|_S$ and $(M,T)|_{S'}$.
\end{theorem-factor uniqueness}
\begin{proof}
If every spherical 2-suborbifold of $(M,T)$ is separating, and if  no edge of $T$ has infinite weight and if $M$ is closed, then this is exactly Petronio's theorem. If $T$ has some edges of infinite weight or if $M$ is not closed, then this can be deduced as in Theorems 8 and 9 of \cite{HAM}. The case when $(M,T)$ contains nonseparating spherical 2-suborbifolds needs to be handled somewhat differently; although the essence of the argument is the same.  In Corollary \ref{thin sum} above we proved that if $(M,T)$ is a nice orbifold, then there exists a system of summing spheres $\wihat{S} \subset (M,T)$. A subset of the components of $\wihat{S}$ will be efficient, as we now explain. Begin by setting $S = \wihat{S}$. Consider  the components of $(M,T)|_S$ one at at time, in some order. Suppose a component is an $\mathbb{S}(0)$ and contains a (necessarily unpunctured) scar but not the matching scar. Let $P \cpt S$ be the sphere producing the scar. Remove $P$ from $S$. The effect on $(M,T)|_S$ is, up to homeomorphism, to remove an $\mathbb{S}(0)$ component. We then restart the process. We handle the $\mathbb{S}(2)$ and $\mathbb{S}(3)$ components similarly. The end result is that after removing some components from $\wihat{S}$ we arrive at a subset $S$ that is efficient.

We now consider uniqueness, adapting the arguments of \cite{Petronio} and \cite{HAM}. We start by establishing some basic moves on systems of spheres. Suppose that $S_0 \subset (M,T)$ is a sphere with $n_0 \leq 2$ punctures and that $S_1 \subset (M,T)$ is a disjoint sphere with $n_0 \leq n_1 \leq 3$ punctures. Let $\alpha$ be an embedded arc joining $S_1$ to $S_0$. If $n_0 = 0$, assume $\alpha$ is disjoint from $T$. If $n_0 = 2$, assume that $\alpha$ is an arc lying in the interior of an edge of $T$. Let $S_2$ be the result of tubing $S_1$ to $S_0$ along $\alpha$. Notice that $S_2$ has the same number of punctures as $S_1$. We say that $S_2$ is obtained by \defn{sliding} $S_1$ over $S_0$ using the \defn{sliding arc} $\alpha$. Observe that $(S_2 \cup S_0)$ is orbifold-homeomorphic to $(S_1 \cup S_0)$ and that $(M,T)|_{(S_2 \cup S_0)}$ is orbifold-homeomorphic to $(M,T)|_{(S_1 \cup S_0)}$. If $S_1$ is essential, then in most cases $S_2$ will also be essential. The only time it is not, is if either:
\begin{itemize}
\item $S_1$ and $S_2$ are parallel unpunctured or twice-punctured spheres in $(M,T)$ and the region of parallelism contains $\alpha$, or
\item  if $S_1$ is thrice-punctured, and the component of $(M,T)|_{S_0 \cup S_1}$ is an $\mathbb{S}(3)$. 
\end{itemize}

Suppose that $S \subset (M,T)$ is an efficient system of summing spheres and that $S_0 \subset (M,T)$ is an essential sphere having $n_0 \leq 2$ punctures which is either a component of $S$ or is disjoint from $S$ and not parallel to a component of $S$. Let $S_1 \cpt S$ be distinct from $S_0$. Suppose that $\alpha$ is a sliding arc such that we may slide $S_1$ over $S_0$ using $\alpha$ and that the interior of $\alpha$ is disjoint from $S$. Let $S_2$ be the the sphere that results from the sliding. By our previous remarks and the definition of efficient, $S_2$ is essential. Furthermore, in $(M,T)|_{S_0}$, the sliding can be achieved by an isotopy of $S_1$. It is then easy to see that $S' = (S \setminus S_1) \cup S_2$ is efficient. 

Suppose that $S, P \subset (M,T)$ are both efficient systems of summing spheres but that either $S$ and $P$ are not orbifold-homeomorphic or $(M,T)|_S$ and $(M,T)|_P$ are not orbifold-homeomorphic. Suppose that there are such choices for $S$ and $P$ so that $S\cap P$ is the union of components of $S$ (equivalently, of $P$). Out of all such choices, choose $S$ and $P$ so that $|S \cap P|$ is maximal. Let $P_0 \cpt P \setminus S$, if such exists. Since $P_0$ is inessential in $(M,T)|_S$ either it bounds a 3-ball in $M|_S$ disjoint from the graph or it is $\boundary$-parallel. Let $B$ be either the 3-ball or the region of parallelism. We may assume that $P_0$ was chosen so that the interior of $B$ is disjoint from $P\setminus S$. 

Consider, first, the possibility that $B$ is a 3-ball; that is $P_0$ is unpunctured. If all the spheres of $S$ with scars in $B$, lie in $S \cap P$, then there is an $\mathbb{S}(0)$ component of $(M,T)|_P$ containing one scar from $P_0$, but not the other. This contradicts the fact that $P$ is efficient. Thus, there is a sphere $S_0 \cpt S \setminus P$ with at least one scar in $B$. Suppose $S_0$ has both scars in $B$. We may join the scars by an arc in $B$ that is disjoint from all other other scars from $S$ in $B$. This arc can be closed up to a loop in $(M,T)|_P$ intersecting $S_0$ exactly once, contradicting the fact that $(M,T)|_P$ is irreducible. Thus, $S_0$ has exactly one scar in $B$. Using arcs in $B$, slide $P_0$ over each sphere of $S$, except for $S_0$, having a scar in $B$. We arrive at the situation where $B$ contains only one scar, and that scar is a scar from $S_0$. Consequently, $P_0$ is then parallel to $S_0$, contradicting our choice of $S$ and $P$ to maximize $|S \cap P|$. Thus, $B$ is not a 3-ball.

Suppose therefore that $B$ is a product region. The boundary of $B$ consists of $P_0\setminus T$ together with a subsurface $F$ of $\boundary (M_0 \setminus T_0)$. If $B$ contains a pair of matching  unpunctured scars, we may slide $P_0$ as before, to remove them. Using arcs in $B$, we may also slide $P_0$ over each unpunctured sphere in $S$  having a single scar in $B$, to remove that scar. Thus, we may assume that $B$ contains no unpunctured scars. If every scar in $B$ (and therefore in $F$) corresponds to a sphere of $S \cap P$, then as before we contradict the assumption that $P$ is efficient. Similarly, if a scar from a sphere of $S \setminus P$ lying in $B$ has its matching scar in $B$, then $(M,T)|_P$ is not irreducible. Suppose that $P_0$ is twice-punctured. This implies that $F$ is an annulus joining the components of $\boundary (P_0 \setminus T)$. Working from either end of $F$, we slide $P_0$ over components of $S \setminus S_0$, until the only scar left in $B$ is that from $S_0$. At this point, $P_0$ is parallel to $S_0$, a contradiction. Thus, $P_0$ is not a twice-punctured sphere.

Suppose that $P_0$ is a thrice-punctured sphere. Using arcs in $B$, we can slide $P_0$ over the twice-punctured spheres with scars in $B$, to arrive that the situation where either there are no scars in $B$ or there is a single thrice-punctured scar in $B$ and no other scars. If there is no scar in $B$, then (after the slides) $P_0$ is inessential in $(M,T)$. By our earlier remarks, this contradicts the fact that $P$ is efficient. Thus, there is a thrice-punctured scar in $B$. Let $S_0 \cpt S$ correspond to it. If $S_0 \cpt P$, then $(M,T)|_P$ contains an $\mathbb{S}(3)$ component having exactly one scar from $P_0$. This contradicts the assumption that $P$ is efficient. Thus, $S_0 \cpt S \setminus P$. Since $P$ is then parallel in $(M,T)$ to $S_0$, we contradict our choice of $S$ and $P$. We conclude that $P \subset S$. If $P \neq S$, then it is easy to see that $S$ cannot be efficient, a contradiction. Thus, $S = P$. This contradicts the hypothesis that either $S$ and $P$ are not orbifold-homeomorphic or that $(M,T)|_S$ and $(M,T)|_P$ are not orbifold-homeomorphic.

As a result of the preceding argument, the result holds when $S \cap P = \nil$. Suppose that $S$ and $P$ are chosen so that they are transverse and $|S \cap P|$ is minimal. Since both $S$ and $P$ are transverse to $T$, $S \cap P$ consists of circles. Choose one $\zeta$ that is innermost in $S$. It bounds an unpunctured or once-punctured disc $D \subset S$ with interior disjoint from $P$. Use $D$ to compress $P$. This separates some component $P_0$ of $P$ into two components $P_1$ and $P_2$. Both of which can be isotoped in $(M,T)$ to be disjoint from $P$. Thus, both are inessential in $(M,T)|_P$. A sequence of slides and isotopies, then allows us to replace $P$ with an efficient system of summing spheres $\wihat{P} \subset (M,T)$ such that $P$ and $\wihat{P}$ are pairwise orbifold-homeomorphic, as are $(M,T)|_P$ and $(M,T)|_{\wihat{P}}$ and for which $|S \cap \wihat{P}| < |S \cap P|$, a contradiction. Thus, it must be the case that an two efficient systems of summing spheres $S$ and $P$ are orbifold-homeomorphic and we also have $(M,T)|_S$ and $(M,T)|_P$ orbifold-homeomorphic.
\end{proof}

\end{document}